\documentclass[12pt,a4paper,twoside,reqno]{amsart}
\usepackage[a4paper,top=3cm,bottom=3cm,inner=2.5cm,outer=2.5cm]{geometry}

\usepackage[utf8]{inputenc}
\usepackage{amsmath,amssymb,amsthm,amsfonts}
\usepackage[english]{babel}
\usepackage{multirow,enumitem,array}
\usepackage{xparse}
\usepackage{pst-node,placeins,xspace,fix-cm}
\usepackage{tikz-cd}
\usepackage{accents,bm}
\usepackage[nomessages]{fp}
\usepackage[colorlinks,linkcolor=blue,citecolor=violet,urlcolor=darkgray,final,hyperindex,linktoc=page,hyperfootnotes=true]{hyperref}
\newcolumntype{F}{>{$}c<{\hspace{-0.9ex}$}}
\newcolumntype{:}{>{$}m{0.8ex}<{$}}
\newcolumntype{R}{>{$}r<{$}}
\newcolumntype{C}{>{$}c<{$}}
\newcolumntype{L}{>{$}l<{$}}
\newcolumntype{N}{@{}>{$}l<{$}}
\newlength\horspace
\newcommand{\h}[1][1.0]{\hspace{#1\horspace}}
\newlength\verspace
\renewcommand{\v}[1][1.0]{\vspace{#1\verspace}\xspace}
\newlength\negverspace
\makeatletter
\newcommand\semiHuge{\@setfontsize\semiHuge{22.72}{27.38}}
\newcommand\semiLarge{\@setfontsize\semiLarge{12}{14}}
\makeatother
\tikzset{iso/.style={draw=none,every to/.append style={edge node={node [sloped, allow upside down, auto=false]{$\cong$}}}}}
\tikzset{simeq/.style={draw=none,every to/.append style={edge node={node [sloped, allow upside down, auto=false]{$\simeq$}}}}}
\tikzset{simeqS/.style={draw=none,every to/.append style={edge node={node [sloped, allow upside down, auto=false]{$\raisebox{0.8em}{$\simeq$}$}}}}}
\tikzset{aiso/.style={simeqS,preaction={draw,->}}}
\tikzset{dotdot/.style={dash pattern=on 0.25ex off 0.2ex, dash phase=0ex}}
\tikzset{RightA/.style={double distance=3.5pt,>={Implies},->},%
	triple/.style={-,preaction={draw,RightA}},%
	quadruple/.style={preaction={draw,RightA,shorten >=0pt},shorten >=1pt,-,double,double distance=0.2pt}}

\newtheorem{teor}{Theorem}[section]
\newtheorem{coroll}[teor]{Corollary}

\newtheorem{prop}[teor]{Proposition}

\theoremstyle{definition}
\newtheorem{defne}[teor]{Definition}

\newtheorem{rem}[teor]{Remark}
\newtheorem{exampl}[teor]{Example}
\newtheorem{rec}[teor]{Recall}
\newtheorem{cons}[teor]{Construction}

\def\nameit#1{\textrm{#1}~}
\def\defx{\nameit{Definition}}
\def\thex{\nameit{Theorem}}
\def\prox{\nameit{Proposition}}
\def\conx{\nameit{Construction}}
\def\remx{\nameit{Remark}}

\def\corx{\nameit{Corollary}}

\def\exax{\nameit{Example}}

\NewDocumentEnvironment{cd}{s O{7} O{7} b}{%
	\IfBooleanF{#1}{\begin{equation*}}\begin{tikzcd}[row sep=#2ex,column sep=#3ex,ampersand replacement=\&]
			#4
		\end{tikzcd}\IfBooleanF{#1}{\end{equation*}}\ignorespacesafterend}{}

\newenvironment{enum}{\begin{enumerate}[label=$($\hspace{0.12ex}\roman*\hspace{0.075ex}$)$]}{\end{enumerate}}

\newenvironment{fun}{\[\begin{tabular}{F:RCL}}{\end{tabular}\]\ignorespacesafterend}
\newenvironment{eqD}[1]{\begin{equation}\label{#1}}{\end{equation}\ignorespacesafterend}
\newenvironment{eqD*}{\begin{equation*}}{\end{equation*}\ignorespacesafterend}

\def\:{\colon}
\providecommand\ordinarycolon{:}
\def\vcentcolon{\mathrel{\mathop\ordinarycolon}}
\newcommand{\deq}{\mathrel{\vcentcolon\mkern-1.2mu=}}
\def\phi{\varphi}

\newcommand{\romanuppercase}[1]{\uppercase\expandafter{\romannumeral #1\relax}}

\newcommand{\p}[1]{\big(\mkern1mu{#1}\mkern1mu\big)}
\newcommand{\pteor}[1]{$($\h[-1.9]{#1}\h[0.6]$)$}

\def\dfn#1{{\itshape #1}}
\def\predfn#1{{\itshape #1}}
\newcommand{\refs}[1]{$($\ref{#1}$)$}

\newcommand{\scaleu}[2][1.2]{{\scalebox{#1}{$#2$}}}

\def\t#1{\widetilde{#1}} 

\def\bl{_{\bullet,\operatorname{lax}}}

\def\c{\circ}

\newcommand{\st}{^{\ast}}

\newcommand{\sliceslant}[2]{{\raisebox{.1em}{$#1$}\mkern-1mu\left/\raisebox{-.25em}{$#2$}\right.}}
\newcommand{\cosliceslant}[2]{{\raisebox{.1em}{$#1$}\mkern-1mu\left/\raisebox{-.14em}{$#2$}\right.}}

\newcommand{\uwidehat}[1]{%
	\mathpalette\douwidehat{#1}%
}
\makeatletter
\newcommand{\douwidehat}[2]{%
	\sbox0{$\m@th#1\widehat{\hphantom{#2}}$}%
	\sbox2{$\m@th#1x$}
	\sbox4{$\m@th#1#2$}
	\dimen0=\ht0
	\advance\dimen0 -.8\ht2
	\dimen2=\dp4
	\rlap{%
		\raisebox{\dimexpr\dimen0-\dimen2}{%
			\scalebox{1}[-1]{\box0}}}{#2}}
\makeatother
\newcommand{\hats}[1]{\widehat{\uwidehat{#1}}}

\DeclareFontFamily{OT1}{pzc}{}
\DeclareFontShape{OT1}{pzc}{m}{it}{<->s*[1.21]pzcmi7t}{}
\DeclareMathAlphabet{\mathpzc}{OT1}{pzc}{m}{it}

\DeclareFontFamily{U}{dutchcal}{\skewchar\font=45}
\DeclareFontShape{U}{dutchcal}{m}{n}{<->s*[1.05] dutchcal-r}{}
\DeclareMathAlphabet{\mathlcal}{U}{dutchcal}{m}{n}
\newcommand{\catfont}[1]{\ensuremath{\mathpzc{#1}}\xspace}

\newcommand{\A}{\catfont{A}}
\newcommand{\B}{\catfont{B}}
\newcommand{\C}{\catfont{C}}
\newcommand{\D}{\catfont{D}}
\newcommand{\E}{\catfont{E}}

\newcommand{\K}{\catfont{K}}
\newcommand{\M}{\catfont{M}}

\newcommand{\Q}{\catfont{Q}}
\newcommand{\V}{\catfont{V}}

\newcommand{\U}{\catfont{U}}

\newcommand{\1}{\catfont{1}}
\newcommand{\2}{\catfont{2}}

\newcommand{\Set}{\catfont{Set}}

\newcommand{\Cat}{\catfont{Cat}}
\newcommand{\CAT}{\catfont{Cat}}
\newcommand{\twoCAT}{2\h[2]\mbox{-}\Cat}
\newcommand{\twoCATlax}{2\h[2]\mbox{-}\Cat_{\h[-1]\lax}}

\newcommand{\VCAT}[1]{\ensuremath{{#1}\mbox{-}\Cat}\xspace}
\newcommand{\twoVCAT}[1]{\ensuremath{2\h[2]\mbox{-}{#1}\mbox{-}\Cat}}

\NewDocumentCommand{\Fib}{t' t" O{n} O{n} o}{
	\ensuremath{\catfont{\ifx#3d{D}\else{\ifx#3o{Op}\else{\ifx#3b{DOp}\else{\ifx#3t{D2Op}\else{\ifx#3c{Cl\h[-3]}\else{\ifx#3s{Sp}\fi}\fi}\fi}\fi}\fi}\fi{Fib}}\IfBooleanTF{#1}{_{\h[0.4]\opn{cart}\ifx#4n{}\else{\h[-1.4],\h[0.4]{#4}}\fi}}{\IfBooleanTF{#2}{_{\h[0.4]\opn{clov}\ifx#4n{}\else{\h[-1.4],\h[0.4]{#4}}\fi}}{\ifx#4n{}\else{_{\h[0.4]{#4}}}\fi}}\IfNoValueF{#5}{\h[-1]\left({#5}\right)}}
}

\NewDocumentCommand{\Sh}{o m}{
	\ensuremath{\catfont{Sh}\hspace{-0.15ex}\left({#2}\IfNoValueF{#1}{,{#1}}\right)}
}

\NewDocumentCommand{\St}{t+ o m}{
	\ensuremath{\catfont{St}\IfBooleanT{#1}{_{\opn{strict}}}\hspace{-0.15ex}\left({#3}\IfNoValueF{#2}{,{#2}}\right)}
}

\NewDocumentCommand{\Alg}{t+ t' m}{
	\ensuremath{\IfBooleanT{#1}{\catfont{Ps}\mbox{-}}{#3}\mbox{-}\catfont{\IfBooleanT{#2}{Co}Alg}}
}

\newcommand{\Lan}[1]{\operatorname{Lan}_{\hspace{0.2ex}#1}}
\newcommand{\lan}[1]{\operatorname{lan}_{\hspace{0.2ex}#1}}

\newcommand{\slice}[2]{\sliceslant{#1}{#2}}
\newcommand{\coslice}[2]{\cosliceslant{#1}{#2}}

\newcommand{\sslash}{\mathbin{/\mkern-6mu/}}
\newcommand{\laxcomma}[2]{{#1}\mkern-1mu\sslash\mkern-1mu{#2}}

\newcommand{\HomC}[3]{{#1}\left({#2},\h[1]{#3}\right)}
\newcommand{\m}[2]{\ensuremath{\left[#1,#2\right]}\xspace}
\newcommand{\mps}[2]{\ensuremath{\left[#1,#2\right]_{\opn{ps}}}\xspace}
\newcommand{\msigma}[2]{\ensuremath{\left[#1,#2\right]_{\opn{sigma}}}\xspace}
\newcommand{\mlax}[2]{\ensuremath{\left[#1,#2\right]_{\lax}}\xspace}

\newcommand{\mlaxn}[2]{\ensuremath{\left[#1,#2\right]_{\laxn}}\xspace}
\newcommand{\moplaxn}[2]{\ensuremath{\left[#1,#2\right]_{\oplaxn}}\xspace}

\newcommand{\opn}[1]{\operatorname{#1}}
\newcommand{\y}[1]{\ensuremath{\operatorname{y}\hspace{-0.2ex}\left({#1}\right)}}

\DeclareMathOperator{\yy}{y}

\newcommand{\id}[1]{\operatorname{id}_{#1}}
\newcommand{\Id}[1]{\operatorname{Id}_{#1}}

\newcommand{\op}{\ensuremath{^{\operatorname{op}}}}
\newcommand{\co}{\ensuremath{^{\operatorname{co}}}}
\newcommand{\coop}{\ensuremath{^{\operatorname{coop}}}}

\newcommand{\ARoplax}[1]{{#1}^{\2}_{\oplax}}

\newcommand{\restr}[2]{{\left.\kern-\nulldelimiterspace {#1}\vphantom{\big|} \right|_{#2}}}
\newcommand{\ceil}[1]{\lceil #1 \rceil}
\newcommand{\dom}{\operatorname{dom}}
\newcommand{\cod}{\operatorname{cod}}

\DeclareMathOperator{\colim}{colim}
\DeclareMathOperator{\lax}{lax}
\DeclareMathOperator{\oplax}{oplax}
\DeclareMathOperator{\laxn}{lax^{cart}}
\DeclareMathOperator{\oplaxn}{oplax^{cart}}
\newcommand{\oplaxnm}[1]{\oplax_{#1}^{\operatorname{cart}}}
\DeclareMathOperator{\opln}{opl^{c}}
\newcommand{\oplaxnopm}[1]{\oplax_{\operatorname{op}\mbox{-}{#1}}^{\operatorname{cart}}}
\newcommand{\wlim}[2]{{\lim}^{#1}\h{#2}}
\newcommand{\wcolim}[2]{{\colim}^{#1}\h{#2}}

\newcommand{\laxnlim}[2]{\laxn\mbox{-}\h[1.5]\wlim{#1}{#2}}
\newcommand{\laxncolim}[2]{\laxn\mbox{-}\h[1.5]\wcolim{#1}{#2}}

\newcommand{\oplaxncolim}[1]{\oplaxn\mbox{-}\h[1.5]\wcolim{\Delta 1}{#1}}
\newcommand{\oplaxnmcolim}[3]{\oplaxnm{#1}\mbox{-}\h[1.5]\wcolim{#2}{#3}}
\newcommand{\oplaxnopmcolim}[3]{{\oplaxnopm{#1}}\mbox{-}\h[1.5]\wcolim{#2}{#3}}

\newcommand{\Int}[1]{\ensuremath{\int \hspace{-0.35ex} #1}}
\newcommand{\Intdiag}[1]{\ensuremath{\scaleu{\int} \hspace{-0.15ex} #1}}
\newcommand{\Intop}[1]{\ensuremath{\int\op \hspace{-0.35ex} #1}}
\newcommand{\Intco}[1]{\ensuremath{\int\co \hspace{-0.35ex} #1}}
\newcommand{\Intcoop}[1]{\ensuremath{\int\coop \hspace{-0.35ex} #1}}
\newcommand{\Intopdiag}[1]{\ensuremath{\scaleu{\int}\op \hspace{-0.15ex} #1}}
\newcommand{\Groth}[1]{\Int{#1}}
\newcommand{\Grothdiag}[1]{\Intdiag{#1}}
\newcommand{\Grothop}[1]{\Intop{#1}}
\newcommand{\Grothco}[1]{\Intco{#1}}
\newcommand{\Grothcoop}[1]{\Intcoop{#1}}
\newcommand{\Grothopdiag}[1]{\Intopdiag{#1}}
\newcommand{\groth}[1]{\mathcal{G}\mkern-1.4mu\left(#1\right)}

\newcommand{\Wlaxn}{\operatorname{W}^{\laxn}}
\newcommand{\Woplaxn}{\operatorname{W}^{\oplaxn}}

\newcommand{\too}{\longrightarrow}
\newcommand{\mto}{\mapsto}
\newcommand{\mtoo}{\longmapsto}

\makeatletter
\newcommand{\ar}[2][]{\xrightarrow[#1]{#2}}
\def\xlongrightarrowfill@{\arrowfill@\relbar\relbar\longrightarrow}
\newcommand{\arr}[2][]{%
	\ext@arrow 0099\xlongrightarrowfill@{#1}{#2}}
\newcommand{\aarr}[2][]{%
	\ext@arrow 0099\xlongrightarrowfill@{#1}{#2}} 

\newcommand{\aR}[2][]{%
	\ext@arrow 0055{\Rightarrowfill@}{#1}{#2}}
\def\xLongrightarrowfill@{\arrowfill@\Relbar\Relbar\Longrightarrow}
\newcommand{\aRR}[2][]{%
	\ext@arrow 0099\xLongrightarrowfill@{#1}{#2}}
\def\aitofill@{\arrowfill@{\lhook\joinrel\relbar}\relbar\rightarrow}
\newcommand{\aito}[2][]{%
	\ext@arrow 3095\aitofill@{#1}{#2}}
\def\Longaitofill@{\arrowfill@{\lhook\joinrel\relbar\joinrel\relbar}\relbar\rightarrow}
\newcommand{\aitoo}[2][]{%
	\ext@arrow 0099\Longaitofill@{#1}{#2}}

\newcommand{\al}[2][]{\xleftarrow[#1]{#2}}
\def\xlongleftarrowfill@{\arrowfill@\longleftarrow\relbar\relbar}
\newcommand{\all}[2][]{%
	\ext@arrow 0099\xlongleftarrowfill@{#1}{#2}}
\newcommand{\aL}[2][]{%
	\ext@arrow 0055{\Leftarrowfill@}{#1}{#2}}
\def\xLongleftarrowfill@{\arrowfill@\Longleftarrow\Relbar\Relbar}
\newcommand{\aLL}[2][]{%
	\ext@arrow 0099\xLongleftarrowfill@{#1}{#2}}
\def\xmapstofill@{\arrowfill@{\mapstochar\relbar}\relbar\rightarrow}
\newcommand{\am}[2][]{%
	\ext@arrow 0395\xmapstofill@{#1}{#2}}
\def\xlongmapstofill@{\arrowfill@\relbar\relbar\longmapsto}
\newcommand{\amm}[2][]{%
	\ext@arrow 0399\xlongmapstofill@{#1}{#2}}

\newcommand{\eqq}{\DOTSB\protect\Relbar\protect\joinrel\Relbar}
\def\xeqqfill@{\arrowfill@\Relbar\Relbar\eqq}
\newcommand{\aeqq}[2][]{%
	\ext@arrow 0099\xeqqfill@{#1}{#2}}
\def\xRrightarrowfill@{\arrowfill@\equiv\equiv\Rrightarrow}
\newcommand{\aM}[2][]{\ext@arrow 0359\xRrightarrowfill@{#1}{#2}}
\newcommand{\Llongrightarrow}{%
	\DOTSB\protect\equiv\protect\joinrel\Rrightarrow}
\def\xLlongrightarrowfill@{\arrowfill@\equiv\equiv\Llongrightarrow}
\newcommand{\aMM}[2][]{%
	\ext@arrow 0099\xLlongrightarrowfill@{#1}{#2}}
\makeatother
\newcommand{\alax}[1]{\aR[\lax]{#1}}
\newcommand{\alaxn}[1]{\aR[\laxn]{#1}}
\newcommand{\aoplaxn}[1]{\aR[\oplaxn]{#1}}

\newcommand{\iso}{\cong}

\newcommand{\aequi}{\ensuremath{\stackrel{\raisebox{-1ex}{\kern-.3ex$\scriptstyle\sim$}}{\rightarrow}}}
\newcommand{\aequii}{\ensuremath{\stackrel{\raisebox{-1ex}{\kern-.3ex$\scriptstyle\sim$}}{\longrightarrow}}}

\newcommand{\PB}[1]{\arrow[#1,phantom,"\scalebox{1.6}{\color{black}$\lrcorner$}",very near start]}
\newcommand{\Ar}[4][]{\arrow[#2,"{#3}"{#1},""{name=#4, anchor=center}]}
\newcommand{\Ars}[4][]{\arrow[#2,"{#3}"'{#1},""{name=#4, anchor=center}]}
\newcommand{\Arb}[6][]{\arrow[#2,"{#3}"{#1},from=#4,to=#5,shorten <= #6 em, shorten >= #6 em]}
\newcommand{\Arbs}[6][]{\arrow[#2,"{#3}"'{#1},from=#4,to=#5,shorten <= #6 em, shorten >= #6 em]}

\NewDocumentCommand{\fib}{O{n} O{2.3} mmm}{%
	\begin{cd}*[#2][5]
		{#3}\ifx#1n{\arrow[d,"{\,\scaleu{#4}}"]}\else{\ifx#1i{\arrow[d,hookrightarrow,"{\,\scaleu{#4}}"]}\else{\ifx#1e{\arrow[d,equal,"{\,\scaleu{#4}}"]}\else{\ifx#1R{\arrow[d,Rightarrow,"{\,\scaleu{#4}}"]}\fi}\fi}\fi}\fi\\
		{#5}\ifx#1o{\arrow[u,"{\,\scaleu{#4}}"']}\fi
	\end{cd}\xspace
}
	
\NewDocumentCommand{\sq}{s O{n} O{7} O{7} O{} O{2.7} O{2.2} O{0.5} O{n}}{%
	\def\foosq##1##2##3##4##5##6##7##8{%
		\IfBooleanTF{#1}{\begin{cd}*}{\begin{cd}}[#3][#4]
				{##1}\ifx#2p{\PB{rd}}\fi\arrow[r,"{##5}"]\ifx#9l{\arrow[d,equal,"{##6}"']}\else{\arrow[d,"{##6}"']}\fi\&{##2}\ifx#9r{\arrow[d,equal,"{##7}"]}\else{\arrow[d,"{##7}"]}\fi\ifx#2l{\arrow[ld,Rightarrow,shorten <=#6ex,shorten >=#7ex,"{#5}"{pos=#8}]}\fi\\
				{##3}\ifx#9d{\arrow[r,equal,"{##8}"']}\else{\arrow[r,"{##8}"']}\fi\ifx#2o{\arrow[ur,Rightarrow,shorten <=#6ex,shorten >=#7ex,"{#5}"{pos=#8}]}\fi\&{##4}
		\end{cd}}%
		\foosq}
\NewDocumentCommand{\sqs}{s O{n} O{7} O{7} O{} O{} O{} O{}}{%
	\def\foosqs##1##2##3##4##5##6##7##8{%
		\IfBooleanTF{#1}{\begin{cd}*}{\begin{cd}}[#3][#4]
				{##1}\ifx#2p{\PB{rd}}\fi\arrow[r,"{##5}"#5]\arrow[d,"{##6}"'#6]\&{##2}\arrow[d,"{##7}"#7]\\
				{##3}\arrow[r,"{##8}"'#8]\&{##4}
		\end{cd}}%
		\foosqs}
\NewDocumentCommand{\nat}{s O{n} O{7} O{7} O{2.7} O{2.2} O{0.5} O{n}}{%
	\def\foonat##1##2##3##4##5##6{%
		\IfBooleanTF{#1}{\sq*}{\sq}[#2][#3][#4][{##1}_{##4}][#5][#6][#7][#8]{{##2}\ifx#8l{}\else{({##5})}\fi}{{##3}\ifx#8r{}\else{({##5})}\fi}{{##2}\ifx#8l{}\else{({##6})}\fi}{{##3}\ifx#8r{}\else{({##6})}\fi}{{##1}_{##5}}{\ifx#8l{}\else{{##2}({##4})}\fi}{\ifx#8r{}\else{{##3}({##4})}\fi}{{##1}_{##6}}}%
	\foonat}
\NewDocumentCommand{\tr}{s O{4.5} O{6.5} O{0} O{0} o}{%
	\def\footr##1##2##3##4##5##6{%
		\IfBooleanTF{#1}{\begin{cd}*}{\begin{cd}}[#3][#2]
				{##1}\arrow[rr,"{##4}"]\IfNoValueTF{#6}{}{\arrow[rr,iso,shift right=#6em]}\arrow[dr,"{##5}"']\&[#4ex]\&[#5ex]{##2}\arrow[ld,"{##6}"]\\
				\&{##3}
		\end{cd}}%
		\footr}
\NewDocumentCommand{\trslice}{s t+ O{7} O{7}}{%
	\def\footrslice##1##2##3##4##5##6{%
		\IfBooleanTF{#1}{\begin{cd}*}{\begin{cd}}[#3][#4]
				{##1}\arrow[r,"{##4}"]\IfBooleanF{#2}{\arrow[d,"{##5}"']}\&{##2}\\
				{##3}\IfBooleanT{#2}{\arrow[u,"{##5}"]}\arrow[ru,"{##6}"']
		\end{cd}}%
		\footrslice}
\NewDocumentCommand{\tc}{s t+ O{7} O{30} O{} O{} O{} o}{
	\def\footc##1##2##3##4##5{%
		\FPmul\Mulresulttwo{#3}{#3}%
		\FPmul\Mulresult{0.0026}{\Mulresulttwo}%
		\IfBooleanTF{#1}{\begin{cd}*}{\begin{cd}}[#3][#3]
				{##1}\Ar[#5]{r,bend left=#4}{##3}{A}\Ars[#6]{r,bend right=#4}{##4}{B}\&{##2}
				\IfBooleanTF{#2}{\Arb[description,pos=0.49]}{\Arb}{Rightarrow #7}{\mkern1mu {##5}}{A}{B}{\IfNoValueTF{#8}{\Mulresult}{#8}}
		\end{cd}}%
		\footc}
\NewDocumentCommand{\tcwl}{s t+ O{7} O{30} O{} O{} O{} o O{-2}}{
	\def\footcwl##1##2##3##4##5##6##7{%
		\FPmul\Mulresulttwo{#3}{#3}%
		\FPmul\Mulresult{0.0026}{\Mulresulttwo}%
		\IfBooleanTF{#1}{\begin{cd}*}{\begin{cd}}[#3][#3]
				##6 \arrow[r,"{##7}"]\&[#9ex]{##1}\Ar[#5]{r,bend left=#4}{##3}{A}\Ars[#6]{r,bend right=#4}{##4}{B}\&{##2}\IfBooleanTF{#2}{\Arb[description,pos=0.49]}{\Arb}{Rightarrow #7}{\mkern1mu {##5}}{A}{B}{\IfNoValueTF{#8}{\Mulresult}{#8}}
		\end{cd}}%
		\footcwl}
\NewDocumentCommand{\tcwr}{s t+ O{7} O{30} O{} O{} O{} o O{-2}}{
	\def\footcwr##1##2##3##4##5##6##7{%
		\FPmul\Mulresulttwo{#3}{#3}%
		\FPmul\Mulresult{0.0026}{\Mulresulttwo}%
		\IfBooleanTF{#1}{\begin{cd}*}{\begin{cd}}[#3][#3]
				{##1}\Ar[#5]{r,bend left=#4}{##3}{A}\Ars[#6]{r,bend right=#4}{##4}{B}\&{##2}\arrow[r,"{##7}"]\&[#9ex]##6\IfBooleanTF{#2}{\Arb[description,pos=0.49]}{\Arb}{Rightarrow #7}{\mkern1mu {##5}}{A}{B}{\IfNoValueTF{#8}{\Mulresult}{#8}}
		\end{cd}}%
		\footcwr}
\NewDocumentCommand{\tcv}{s t' O{7} O{30} mmmmm}{
	\FPmul\Mulresulttwo{#3}{#3}%
	\FPmul\Mulresult{0.0026}{\Mulresulttwo}%
	\IfBooleanTF{#1}{\begin{cd}*}{\begin{cd}}[#3][#3]
			{#5}\IfBooleanTF{#2}{\Ars{d,leftarrow,bend right=#4}{#7}{A}\Ar{d,leftarrow,bend left=#4}{#8}{B}}{\Ars{d,bend right=#4}{#7}{A}\Ar{d,bend left=#4}{#8}{B}}\\{#6}
			\Arb{Rightarrow}{#9}{A}{B}{\Mulresult}
		\end{cd}}
\NewDocumentCommand{\twonats}{s O{2.2} O{8} O{7} O{1.05} O{3.45} O{2}}{%
	\def\footwonats##1##2##3##4##5##6##7##8##9{%
		\def\foofootwonats####1####2####3####4####5{%
			\IfBooleanTF{#1}{\begin{cd}*}{\begin{cd}}[#3][#4]
					##1 \Ar{r}{##9}{} \Ars{d,bend right=40}{##5}{A} \Ar{d,bend left=40}{##6}{B} \&
					##2 \Ars{d,bend left}{##8}{Q} \arrow[ld,Rightarrow,shift left=#7,"{####4}"{pos=0.48},shorten <=#5ex, shorten >=#6ex]\&[-2ex]
					##1 \Ar{r}{##9}{} \Ar{d,bend right}{##5}{R} \&
					##2 \Ars{d,bend right=40}{##7}{C} \Ar{d,bend left=40}{##8}{D} \arrow[ld,Rightarrow,shift right=#7,"{####5}"'{pos=0.52},shorten <=#6ex, shorten >=#5ex] \\
					##3 \Ars{r}{####1}{} \&
					##4 \&
					##3 \Ars{r}{####1}{} \& 
					##4
					\Arbs{Rightarrow}{\,{####2}}{B}{A}{0.3}
					\Arbs{Rightarrow}{\,{####3}}{D}{C}{0.3}
					\Arb{equal}{}{Q}{R}{#2}
			\end{cd}}%
			\foofootwonats}\footwonats}

\begin{document}

\title[Pointwise Kan extensions along 2-fibrations and 2-category of elements]{Pointwise Kan extensions along 2-fibrations\\ and the 2-category of elements}
\author[L. Mesiti]{Luca Mesiti}
\address{School of Mathematics, University of Leeds}
\email{mmlme@leeds.ac.uk}

\begin{abstract}
	We study the 2-category of elements from an abstract point of view. We generalize to dimension 2 the well-known result that the category of elements can be captured by a comma object that also exhibits a pointwise left Kan extension. For this, we propose an original definition of pointwise Kan extension along a discrete 2-opfibration in the lax 3-category of 2-categories, 2-functors, lax natural transformations and modifications. Such definition uses cartesian-marked lax limits, which are an alternative to weighted 2-limits. We show that a pointwise Kan extension along a discrete 2-opfibration is always a weak one as well. The proof is based on an original generalization of the parametrized Yoneda lemma which is as lax as it can be.
\end{abstract}

\maketitle

\setcounter{tocdepth}{1}
\tableofcontents

\section{Introduction}

In this paper we study the $2$-category of elements from an abstract point of view. This is the $2$-dimensional generalization of the construction of the category of elements, and it has been introduced by Street in~\cite{street_limitsindexedbycatvalued}. It is at the same time a natural extension of the usual Grothendieck construction that admits $2$-functors from a $2$-category $\B$ into $\CAT$, and a restriction of the $2$-dimensional Grothendieck construction of Bakovi\'c~\cite{bakovic_fibrofbicat} and Buckley~\cite{buckley_fibredtwocatandbicat} to $2$-functors into $\twoCAT$ that factor through $\CAT$. Analogously, the corresponding notion of opfibration, introduced by Lambert in~\cite{lambert_discretetwofib} with the name discrete $2$-opfibration, is at the same time a natural extension of the usual Grothendieck opfibrations and a locally discrete version of Hermida's~\cite{hermida_somepropoffibasafibredtwocat} $2$-fibrations. Lambert proved in~\cite{lambert_discretetwofib} that discrete $2$-opfibrations with small fibres form the essential image of the $2$-functor that calculates the $2$-category of elements. We extend this result to $2$-equivalences between $2$-copresheaves and discrete $2$-opfibrations (\thex\ref{twofullyfaithfulnessgroth}).

Capturing the $2$-category of elements from an abstract point of view will be very useful for generalizations of the Grothendieck construction. In future work, we would like to achieve the generalization to the enriched context. While Vasilakopoulou defined in~\cite{vasilakopoulou_enrichedfibrations} a notion of fibration enriched over a monoidal fibration, it seems no one has yet proposed a good notion of fibration for a $\V$-enriched functor, with $\V$ a nice enough monoidal category. There would be many applications of such enriched fibrations. Some examples would be additive fibrations, graded fibrations, metric fibrations and general quantale-enriched fibrations. We expect this paper to be useful towards such a theory. As giving an explicit definition of enriched fibration is quite hard, we would like to capture Grothendieck fibrations and the Grothendieck construction from an abstract point of view and try to generalize such abstract theory.

Another motivation that we have in mind is to understand how the various properties of the $2$-category of elements are connected with each other. We will show in Section~\ref{sectiontwosetgrothconstr} that our pointwise Kan extension result for the $2$-category of elements (see below) implies many other properties. Among these, the conicalization of weighted $2$-limits and the $2$-fully faithfulness of the $2$-functor that calculates the $2$-category of elements.

In dimension $1$, it is known that the category of elements can be captured in a more abstract way. Given a copresheaf $F\:\B\to \Set$, the construction of the category of elements of $F$ is equivalently given by the comma object 
\begin{eqD*}
	\sq*[l][5][5][\opn{comma}][2.7][2.2][0.6]{\Grothop{F}}{\1}{\B}{\Set}{}{\groth{F}}{1}{F}
\end{eqD*}
Moreover, this filled square exhibits $F$ as the pointwise left Kan extension of the constant at $\1$ functor $\Delta 1$ along the discrete opfibration $\groth{F}$.

Our main theorem (\thex\ref{teorpointkanextforgroth}, after \thex\ref{grothconstrislaxcomma}) is a $2$-dimensional generalization of this result. We prove that an analogous square as the above one exhibits, at the same time, the $2$-category of elements $\groth{F}$ as a lax comma object in $\twoCATlax$ and $F$ as the pointwise left Kan extension in $\twoCATlax$ of $\Delta 1$ along the discrete $2$-opfibration $\groth{F}$. We find the need to consider the lax $3$-category $\twoCATlax$ of $2$-categories, $2$-functors, lax natural transformations and modifications, because the analogue of the square above is now only filled by a lax natural transformation. Lax $3$-categories, introduced by Lambert in~\cite{lambert_discretetwofib}, are categories enriched over the cartesian closed category of $2$-categories and lax functors. As $\twoCATlax$ has not yet been studied much, we have to propose a novel definition of pointwise Kan extension in $\twoCATlax$ (\defx\ref{defpointkanextensionlaxthreecat}) to achieve our objective. We then prove that pointwise Kan extensions in $\twoCATlax$ along a discrete $2$-opfibration are always weak ones as well (\prox\ref{pointkanisalsoweak}). The proof is based on an original generalization of the parametrized Yoneda lemma which is as lax as it can be (\thex\ref{parametrizedyonedaoplaxnlax}). 

Pointwise Kan extensions are actively researched. While it is relatively easy to give notions of weak Kan extension in a categorical framework, it is much harder to give the corresponding pointwise notions. In~\cite{street_fibandyonlemma}, Street proposes to look at the stability of a Kan extension under pasting with comma objects to obtain a definition of pointwise Kan extension in any $2$-category. However, applying Street's definition to the $2$-category $\VCAT{\V}$ does not give the right notion. For enriched $\V$-functors, the correct notion, that uses $\V$-limits, has been introduced by Dubuc in~\cite{dubuc_kanextensionsinenrichedcat} and later used by Kelly in~\cite{kelly_basicconceptsofenriched}. However, we needed pointwise Kan extensions in the lax $3$-category $\twoCATlax$, that we view as $\twoVCAT{\Set}$, with an original idea of $2$-$\V$-enrichment. Pseudo-Kan extensions of Lucatelli Nunes's~\cite{lucatellinunes_biadjointtriangles} are a pseudo version of Dubuc's ones, considering weighted bilimits in the place of weighted $2$-limits, and quasi-Kan extensions of Gray's~\cite{gray_quasikanextensions} are a lax version. Instead, we needed a strict version, using some form of strict $2$-limit, but which also takes the $2$-$\V$-enrichment into account.

To give our definition of pointwise Kan extension in $\twoCATlax$, we take advantage of the connection between the $2$-category of elements and cartesian-marked oplax colimits, that we recall in Section~\ref{sectionlaxnormallimits} from an original and more elementary perspective. Cartesian-marked (op)lax conical (co)limits are a particular case of a $2$-dimensional notion of limit introduced by Gray in~\cite{gray_formalcattheory}. As proved by Street in~\cite{street_limitsindexedbycatvalued}, and here as well with a more elementary proof, they are an alternative to weighted $2$-limits. Indeed they are particular weighted $2$-limits and every weighted $2$-limit can be reduced to one of them. But cartesian-marked lax conical limits can be much more useful in some situations, as they allow to consider cones rather than cylinders, up to filling the cones with coherent $2$-cells. Some of these $2$-cells are required to be the identity, whence the adjective ``marked". As the choice of such $2$-cells comes from the cartesian liftings of a $2$-category of elements, we call them ``cartesian". Despite their potential, cartesian-marked lax limits have been almost forgotten, until Descotte, Dubuc and Szyld's paper~\cite{descottedubucszyld_sigmalimandflatpseudofun}, where they use their pseudo version, called by them \predfn{sigma-limits}. Later, in~\cite{szyld_limlifttheortwodimmonad} and~\cite{szyld_liftingpielimitswithstrictproj}, Szyld also considered the strict version that we use here. In our~\cite{mesiti_colimitsintwodimslices}, the reduction of weighted $2$-limits to \predfn{cartesian-marked lax conical} ones allowed us to develop a calculus of colimits in $2$-dimensional slices. In our~\cite{mesiti_twoclassifiersdensegenstacks}, it then allowed us to reduce the study of a $2$-classifier to dense generators and to construct a good $2$-classifier in stacks, generalizing to dimension 2 the fundamental result that a Grothendieck topos is an elementary topos.

This paper has also potential applications to higher dimensional elementary topos theory. Indeed, according to Weber's~\cite{weber_yonfromtwotop}, the filled square drawn above presents $\CAT$ as the archetypal $2$-dimensional elementary topos. The classification process is the category of elements construction. On this line, we believe we should consider $\twoCATlax$ as the archetypal $3$-dimensional elementary topos. Its classifier would be $\1\:\1\to \CAT$ and its classification process would be the $2$-category of elements construction. We inscribe the sequence of elementary $n$-topoi
$$\Set\quad\leadsto\quad \CAT \quad\leadsto\quad \twoCATlax$$
in an original idea of \predfn{$2$-$\V$-enrichment}. This guided our definition of pointwise Kan extension in $\twoCATlax$. And we believe that this observation could be useful also towards an enriched version of the Grothendieck construction. For this, we may also notice that comma objects, that regulate the classification process in $\CAT$, are the archetypal example of exact square in $\CAT$. And the fact that every copresheaf is the pointwise Kan extension of $\Delta 1$ along its category of elements is actually a consequence of having an exact square, together with $\1\:\1\to \Set$ being dense. Moving to $\twoCATlax$, we need to upgrade comma objects to lax comma objects. We give a new, refined universal property of lax comma objects to suit the lax $3$-dimensional ambient of $\twoCATlax$ (\defx\ref{completeunivproplaxcomma}). This improves both the universal properties given by Gray in~\cite{gray_formalcattheory} and by Lambert in~\cite{lambert_discretetwofib}.

To clarify references to an earlier draft of this work from other papers, we should mention that the 2-category of elements had been initially called the $2$-$\Set$-enriched Grothendieck construction and that cartesian-marked lax limits had been initially called lax normal limits.

\subsection*{Outline of the paper}

In Section~\ref{sectionlaxnormallimits}, we recall from an original perspective the explicit $2$-category of elements and the cartesian-marked (op)lax conical (co)limits. We give a new, more elementary proof of the equivalence between weighted $2$-limits and cartesian-marked lax conical ones.

In Section~\ref{sectionkanextensions}, we present an original definition of pointwise Kan extension in $\twoCATlax$ along a discrete $2$-opfibration. We prove that such a pointwise Kan extension is always a weak one as well. The proof is based on an original generalization of the parametrized Yoneda lemma.

In Section~\ref{sectiontwosetgrothconstr}, we generalize to dimension $2$ the fact that the category of elements can be captured by a comma object that also exhibits a pointwise left Kan extension. For this we use our notion of pointwise Kan extension in $\twoCATlax$ and a new refined notion of lax comma object in $\twoCATlax$.

\section{Cartesian-marked (op)lax conical (co)limits}
\label{sectionlaxnormallimits}

In this section we recall from an original perspective the explicit $2$-category of elements construction, introduced by Street in~\cite{street_limitsindexedbycatvalued}, and the cartesian-marked (op)lax conical (co)limits, a particular case of a notion introduced by Gray in~\cite{gray_formalcattheory}. We originally show how the two concepts arise simultaneously by the wish of giving an essential solution to the problem of conicalization of the weighted $2$-limits. This can also be seen as a justification to both the concepts.

We obtain a new, more elementary proof of the fact (firstly proved by Street in~\cite{street_limitsindexedbycatvalued}) that weighted $2$-limits and \predfn{cartesian-marked lax conical limits} give equivalent theories (\thex\ref{redlaxnormalconical} and \thex\ref{laxnormalareweighted}). 

 Cartesian-marked lax conical limits offer huge benefits in many situations, as they have a conical shape, even if with coherent $2$-cells inside the triangles that form the cone. Sometimes, it is much easier to handle such $2$-cells rather than a non-conical shape. See for example our~\cite{mesiti_colimitsintwodimslices} and~\cite{mesiti_twoclassifiersdensegenstacks}, as explained in the Introduction. 

In Section~\ref{sectionkanextensions}, the idea of cartesian-marked (op)lax (co)limit will guide the original definitions of colimit in a \predfn{$2$-$\Set$-category} and of \predfn{pointwise left Kan extension in $\twoCATlax$}.

 \begin{rec}[for example from Kelly's~\cite{kelly_basicconceptsofenriched}]
 	Let $\V$ be a complete and cocomplete symmetric closed monoidal category. Consider $\V$-functors $F\:\A \to \C$ (the diagram) and $W\:\A\to \V$ (the weight), with $\A$ a small $\V$-category. The \dfn{$\V$-limit of $F$ weighted by $W$}, denoted as $\wlim{W}{F}$, is (if it exists) an object $L\in \C$ together with an isomorphism in \V
 	\begin{equation}\label{isoweightlim}
 		\HomC{\C}{U}{L}\iso \HomC{\m{\A}{\V}}{W}{\HomC{\C}{U}{F(-)}}
 	\end{equation}
 	$\V$-natural in $U\in \C\op$, where $\m{\A}{\V}$ is the $\V$-category of $\V$-copresheaves on $\A$ valued in $\V$ enriched over itself. When $\wlim{W}{F}$ exists, the identity on $L$ provides a $\V$-natural transformation $\lambda\:W\aR{}\HomC{\C}{L}{F(-)}$ called the \dfn{universal cylinder}.
 	
 	For the notion of \dfn{weighted $\V$-colimit}, we start from $F\:\A\to\C$ and $W\:\A\op\to\V$. The universal property is then
 	\begin{equation*}
 		\HomC{\C}{\wcolim{W}{F}}{U}\iso \HomC{\m{\A\op}{\V}}{W}{\HomC{\C}{F(-)}{U}}
 	\end{equation*}
 	
	Although the classical constant at $1$ weight $\Delta 1$, called the \dfn{conical weight}, no longer suffices in the general enriched setting, we pay attention to when a weighted limit can be \dfn{reduced to a conical one}. It is well known (see Kelly's~\cite{kelly_basicconceptsofenriched}) that in the $\Set$-enriched setting every weighted limit can be conicalized, using the category of elements. A general strategy of conicalization (used by Kelly in~\cite{kelly_basicconceptsofenriched}) allows to conicalize just the $\V$-colimits $W\iso \wcolim{W}{\yy}$, for $W\:\A\to \V$ an enriched presheaf with $\A$ small and $\yy\:\A\op\to\m{\A}{\V}$ the $\V$-Yoneda embedding. The lemma of continuity of a limit in its weight then allows to deduce that all weighted limits are conicalized. The formula is
	\begin{equation}\label{isolemmacont}
		\wlim{\wcolim{W}{H}}{F} \iso \wlim{W}{\left(\wlim{H(-)}{F}\right)}
	\end{equation}
\end{rec}

\begin{rem}\label{needoflaxcones}
	When $\V=\Cat$, the conicalization of weighted $2$-limits is, strictly speaking, not possible. We would need to encode the universal cocylinder $\mu$ exhibiting $W\iso \wcolim{W}{\yy}$ (given by the Yoneda lemma) in terms of a universal cocone
	$$\t{\mu}\:\Delta 1\aR{}\HomC{\m{\A}{\Cat}}{H(-)}{W}$$
	with $H$ some $2$-functor $\B\to \m{\A}{\Cat}$.  And this is not possible because the components of $\mu$ are functors rather than mere functions. The idea is to admit $2$-cells inside the cocone $\t{\mu}$ in order to encode the extra data. We thus consider lax natural transformations, that have general structure 2-cells inside the naturality squares. But, as we explain in \conx\ref{2SetGrothconstrtoconicalize}, the right notion of relaxed 2-natural transformation to consider will be a marked version of lax natural transformations.
\end{rem}

\begin{cons}[$2$-category of elements]\label{2SetGrothconstrtoconicalize}
	Following \remx\ref{needoflaxcones}, we search for a relaxed notion of 2-natural transformation, which is however stronger than a lax natural transformation, and for a $2$-functor $H\:\B\to \m{\A}{\CAT}$ such that any cocylinder
		$$\phi\:W\aR{}\HomC{\m{\A}{\CAT}}{\y{-}}{U}$$
		with $U\:\A\to \CAT$ can be encoded in terms of a relaxed $2$-natural transformation
		$$\t{\phi}\:\Delta 1 \aR[\opn{relaxed}]{} \HomC{\m{\A}{\CAT}}{H(-)}{U}\:\B\op\to\CAT.$$
		
		In order to deduce that every weighted $2$-limit can be analogously conicalized via the lemma of continuity of a limit in its weight, we need $H$ of the form
		$${\left(\Grothopdiag{W}\right)}\op\arr{{\groth{W}}\op} \A\op\arr{\yy} \m{\A}{\CAT}.$$
		Up to now, $\Grothop{W}$ and $\groth{A}$ are just symbols, but will be found to be the \predfn{$2$-category of elements}, as defined explicitly by Street in~\cite{street_limitsindexedbycatvalued}. The explicit formulas that we find for building $\t{\phi}$ from $\phi$ are stricter analogues of the ones of Descotte, Dubuc and Szyld's~\cite{descottedubucszyld_sigmalimandflatpseudofun} (where they essentially conicalize bilimits). 
	
		For every $A\in \A$ and $X\in W(A)$, we have a morphism $\phi_A(X)\:\y{A}\to U$, and we want to form the cocone $\t{\phi}$ exactly with these morphisms. So we take the objects of $\Grothop{W}$ to be all pairs $(A,X)$ with $A\in \A$ and $X\in W(A)$, and define $\groth{W}(A,X)\deq A$. We then set $\t{\phi}_{(A,X)}\deq \phi_A(X)$.
		
		But $\t{\phi}$ also needs to encode the assignment of every $\phi_A$ on morphisms $\alpha\:X\to X'$ in $W(A)$. Lax naturality of $\t{\phi}$ allows to have, for every $\xi\:(A,X) \to (A',X')$ in $\Grothop{W}$, a $2$-cell
		\begin{cd}[4][6]
			1 \arrow[r,"{\t{\phi}_{(A,X)}}"] \arrow[d,equal]\& \HomC{\m{\A}{\CAT}}{\y{A}}{U}\arrow[d,"{-\c \y{\h\groth{W}\op(\xi)}}"]\arrow[dl,Rightarrow,"{\t{\phi}_\xi}"'{pos=0.43}, shorten <= 2.8ex, shorten >= 4.9ex]\\
			1 \arrow[r,"{\t{\phi}_{(A',X')}}"'] \& \HomC{\m{\A}{\CAT}}{\y{A'}}{U}
		\end{cd} 
		For every $A\in \A$ and $\alpha\:X\to X'$ in $W(A)$, we need a morphism $(A,X)\to (A,X')$ in $\Grothop{W}$ whose image with respect to $\groth{W}$ is $\id{A}$. Wishing to write the action of $\groth{W}$ as a projection on the first component, we call such morphism $(A,X)\to (A,X')$ as $(\id{A},\alpha)$. We set $\t{\phi}_{(\id{A},\alpha)}\deq \phi_A(\alpha)$.
		
		We now encode the $2$-naturality of $\phi$ into the relaxed naturality of $\t{\phi}$. For every $f\:A\to A'$ in $\A$ and $X\in W(A)$, the naturality of $\phi$ expresses the equality
		$$\phi_{A'}(W(f)(X))=\phi_A(X)\c \y{f}.$$
		So, for every $f\:A\to A'$ in $\A$ and $X\in W(A)$, we need a morphism
		$$\underline{f}^X\:(A,X)\to (A',W(f)(X))$$
		in $\Grothop{W}$ such that $\groth{W}\left(\underline{f}^X\right)=f$ and $\t{\phi}_{\underline{f}^X}=\id{}$.
		
		It is natural to take $\underline{\id{A}}^X=(\id{A},\id{X})$ for every $A\in \A$ and $X\in W(A)$ and ask any of such equal morphisms to be the identity on $(A,X)$. We then need to close the union of the two kinds of morphisms $(\id{A},\alpha)$ and $\underline{f}^X$ under composition. For this, we notice that, given $f\:A\to A'$ in $\A$ and $\alpha\:X\to X'$ in $W(A)$, the two morphisms $\underline{f}^{X'}\c (\id{A},\alpha)$ and $(\id{A'},W(f)(\alpha))\c \underline{f}^X$ in $\Grothop{W}$ will have the same associated structure $2$-cell of $\t{\phi}$, by lax naturality of $\t{\phi}$. We then take such two morphisms in $\Grothop{W}$ to be equal, so that we will be able to recover the naturality of $\phi$ (on morphisms) starting from $\t{\phi}$. At this point, every finite composition of morphisms in $\Grothop{W}$ can be reduced to a composite\v[-1]
		$$(A,X)\ar{\underline{f}^X} (A',W(f)(X)) \ar{(\id{A'},\alpha)} (A',X')$$
		for some $f\:A\to A'$ in $\A$ and $\alpha\:W(f)(X)\to X'$ in $W(A')$. We define the morphisms in $\Grothop{W}$ to be all the formal composites $(\id{A'},\alpha)\c \underline{f}^X$, that we call $(f,\alpha)$. And we see that $\underline{f}^X=(f,\id{W(f)(X)})$. Functoriality forces $\groth{W}(f,\alpha)=f$, and lax naturality forces
		$$\t{\phi}_{(f,\alpha)}=\t{\phi}_{(\id{A'},\alpha)\c (f,\id{})}=\phi_{A'}(\alpha)\c \id{}=\phi_{A'}(\alpha).$$
		
		We now want to encode the $2$-dimensional part of the $2$-naturality of $\phi$ into the $2$-dimensional part of the relaxed naturality of $\t{\phi}$. Lax naturality of $\t{\phi}$ allows to have, for every $2$-cell $\Xi\:(f,\alpha)\aR{}(g,\beta)\:(A,X)\to (A',X')$ in $\Grothop{W}$,
		\begin{eqD}{twodimlaxphitilde}
			\scalebox{0.9}{\begin{cd}*[6.5][7]
				1 \arrow[r,"{\t{\phi}_{(A,X)}}"] \arrow[d,equal]\& \HomC{\m{\A}{\CAT}}{\y{A}}{U}\arrow[d,"{-\c \y{f}}"]\arrow[dl,Rightarrow,"{\t{\phi}_{(f,\alpha)}}"{pos=0.445}, shorten <= 2.8ex, shorten >= 4.9ex]\\
				1 \arrow[r,"{\t{\phi}_{(A',X')}}"'] \& \HomC{\m{\A}{\CAT}}{\y{A'}}{U}
			\end{cd} = \quad
			\begin{cd}*[6.5][7]
				1 \arrow[r,"{\t{\phi}_{(A,X)}}"] \arrow[d,equal]\& 	\HomC{\m{\A}{\CAT}}{\y{A}}{U}\arrow[d,bend right,shift right=0ex,"{-\c \y{g}}"'{name=B}]\arrow[d,bend left,shift left = 2ex,"{-\c \y{f}}"{name=A}]\arrow[dl,Rightarrow,shift right=1ex,"{\t{\phi}_{(g,\beta)}}"{pos=0.55}, shorten <= 4.65ex, shorten >= 3.2ex]\\
				1 \arrow[r,"{\t{\phi}_{(A',X')}}"'] \& \HomC{\m{\A}{\CAT}}{\y{A'}}{U}
				\arrow[Rightarrow,from=A,to=B,"{-\ast \y{\h\groth{W}\op(\Xi)}}"'{inner sep=1ex},shorten <=1.55ex, shorten >=1.5ex]
			\end{cd}}
		\end{eqD}
		The $2$-naturality of $\phi$ expresses the following equality, for every $2$-cell $\delta\:f\aR{}g\:A\to A'$ in $\A$ and for every $X\in W(A)$:
		$$\phi_{A'}(W(\delta)_X)=\phi_A(X)\h\y{\delta}\:\phi_{A'}(W(f)(X))\aR{}\phi_{A'}(W(g)(X)).$$
		So, for every $2$-cell $\delta\:f\aR{}g\:A\to A'$ in $\A$ and every $X\in W(A)$, we need a $2$-cell in $\Grothop{W}$, that we call $\underline{\delta}^X$ or just $\delta$, such that
		$$\underline{\delta}^X\:(f,W(\delta)_X)\aR{}(g,\id{})\:(A,X)\to (A',W(g)(X))$$
		and $\groth{W}(\underline{\delta}^X)=\delta$. These $2$-cells are closed under both vertical and horizontal composition, inherited from $\A$, but we have to close them under whiskering with morphisms $(\id{A},\alpha)$. Notice that for every $\delta\:f\aR{}g\:A\to A' $ in $\A$ and every $\alpha\:X\to X'$ in $W(A)$, we have that the axiom of equation~\refs{twodimlaxphitilde} of $\t{\phi}$ on the two whiskerings $\underline{\delta}^{X'}\h(\id{A},\alpha)$ and $(\id{A'},W(g)(\alpha))\h \underline{\delta}^{X}$ in $\Grothop{W}$ is exactly the same. So we ask such two whiskerings in $\Grothop{W}$ to be equal. At this point, every horizontal composition of $2$-cells in $\Grothop{W}$ can be reduced to a whiskering of the form $(\id{},\beta)\h \underline{\delta}^X$ for some $2$-cell $\delta\:f\aR{}g\:A\to A'$ in $\A$, $X\in W(A)$ and $\beta\:W(g)(X)\to X'$ in $W(A')$. We define the $2$-cells in $\Grothop{W}$ to be precisely such formal whiskerings. Equivalently, a $2$-cell $(f,\alpha)\aR{}(g,\beta)\:(A,X)\to (A',X')$ in $\Grothop{W}$ is a $2$-cell $\delta\:f\aR{}g$ in $\A$ such that
		$$\alpha=\beta\c W(\delta)_X.$$
		For this, we will call such $2$-cell just $\delta\:(f,\alpha)\aR{}(g,\beta)$. Compositions are inherited from $\A$. Then $2$-functoriality forces $\groth{W}(\delta)=\delta$.
		
		It is straightforward to show that $\Grothop{W}$ is a $2$-category and that $\groth{W}\:\Grothop{W}\to \A$ is a $2$-functor. Notice that we have also described the right notion of relaxed $2$-natural transformation that $\t{\phi}$ needs to satisfy to encode the $2$-naturality of $\phi$. It is a form of marked lax natural transformation, i.e.\ a lax natural transformation such that certain structure $2$-cells are asked to be identities (\defx\ref{deflaxnormal}).
	\end{cons}

We read from \conx\ref{2SetGrothconstrtoconicalize} the following explicit definition of the \predfn{$2$-category of elements}, that coincides with the one of Street's~\cite{street_limitsindexedbycatvalued}.

\begin{defne}\label{defexpltwosetgroth}
	Let $F\:\B\to\CAT$ be a $2$-functor with $\B$ a $2$-category. The \dfn{$2$-category of elements of $F$} is the $2$-functor $\groth{F}\:\Grothop{F}\to \B$, given by the projection on the first component, with $\Grothop{F}$ such that:
	\begin{description}
		\item[an object of $\Grothop{F}$] is a pair $(B,X)$ with $B\in \B$ and $X\in F(B)$;
		\item[a morphism $(B,X)\to (C,X')$ in $\Grothop{F}$] is a pair $(f,\alpha)$ with $f\:B\to C$ a morphism in $\B$ and $\alpha\:F(f)(X)\to X'$ a morphism in $F(C)$;
		\item[a $2$-cell $(f,\alpha)\aR{}(g,\beta)\:(B,X)\to (C,X')$ in $\Grothop{W}$] is a $2$-cell $\delta\:f\aR{}g$ in $\B$ such that $\alpha=\beta\c F(\delta)_X$;
		\item[the compositions and identities] are as described in \conx\ref{2SetGrothconstrtoconicalize}.
	\end{description}
\end{defne}
	
	\begin{defne}[Lambert~\cite{lambert_discretetwofib}]\label{deftwosetopf}
		Let $\B$ be a $2$-category. A \dfn{discrete $2$-opfibration over $\B$} is a $2$-functor $P\:\E\to \B$ such that
		\begin{enum}
			\item the underlying functor $P_0$ of $P$ is an ordinary Grothendieck opfibration;
			\item for every pair $X,Y\in \E$ the functor $P_{X,Y}\:\HomC{\E}{X}{Y}\to \HomC{\B}{P(X)}{P(Y)}$ is a discrete fibration.
		\end{enum}
		We say that $P$ is \dfn{split} if $P_0$ is so.
	\end{defne}
	
	\begin{teor}[Lambert~\cite{lambert_discretetwofib}] \label{essentialimagegivenbytwosetopf}
		Let $\B$ be a $2$-category. The essential image of the $2$-functor
		$$\groth{-}\:\m{\B}{\CAT}\to \slice{\twoCAT}{\B}$$
		is given by the split discrete $2$-opfibrations with small fibres.
	\end{teor}
	
		We will extend \thex\ref{essentialimagegivenbytwosetopf} to a complete $2$-equivalence between $\m{\B}{\CAT}$ with various laxness flavours on morphisms and corresponding $2$-categories of discrete $2$-opfibrations in Section~\ref{sectionkanextensions}.
	
	From \conx\ref{2SetGrothconstrtoconicalize} we also obtain the following definition.
	
	\begin{defne}\label{deflaxnormal}
		Let $W\:\A\to \CAT$ be a $2$-functor with $\A$ small, and consider $2$-functors $M,N\:\Grothop{W}\to \D$. A \dfn{cartesian-marked lax natural transformation $\alpha$ from $M$ to $N$}, denoted $\alpha\:M\alaxn{}N$,\v is a lax natural transformation $\alpha$ from $M$ to $N$ such that the structure $2$-cell on every morphism of the form $\p{f,\id{W(f)(X)}}\:(A,X)\to (B,W(f)(X))$ in $\Grothop{W}$ is the identity.
	\end{defne}

\begin{rem}\label{motivationlaxnormal}
	Cartesian-marked lax natural transformations are a particular case of a more general notion of marked lax natural transformation introduced by Gray in~\cite{gray_formalcattheory}. They are a stricter analogue of Descotte, Dubuc and Szyld's~\cite{descottedubucszyld_sigmalimandflatpseudofun} sigma limits. ``Cartesian" refers to the fact that the marking is precisely given by the chosen cartesian liftings of the 2-category of elements $\groth{W}\:\Grothop{W}\to \A$. Street showed in~\cite{street_limitsindexedbycatvalued} that this less general notion is totally sufficient to build all the general limits considered by Gray.
\end{rem}
	
	\begin{defne}\label{laxnormalconicallimits}
		Let $W\:\A\to \CAT$ be a $2$-functor with $\A$ small, and let $F\:\Grothop{W} \to \C$ be a $2$-functor. Notice that $\Grothop{W}$ is small, since $\A$ is small. The \dfn{cartesian-marked lax conical limit of $F$}, denoted as $\laxnlim{\Delta 1}{F}$, is (if it exists) an object $L\in \C$ together with an isomorphism of categories
		$$\HomC{\C}{U}{L}\iso \HomC{\mlaxn{\Grothopdiag{W}}{\CAT}}{\Delta 1}{\HomC{\C}{U}{F(-)}}$$
		$2$-natural in $U\in \C\op$,\v where $\mlaxn{\Grothop{W}}{\CAT}$ is the $2$-category of $2$-functors, cartesian-marked lax natural transformations and modifications from $\Grothop{W}$ to \CAT.
		
		\noindent When $\laxnlim{\Delta 1}{F}$ exists, the identity on $L$ provides a cartesian-marked lax natural transformation $\lambda\:\Delta 1 \alaxn{}\HomC{\C}{L}{F(-)}$ called the \dfn{universal cartesian-marked lax cone}.
				
		Let $W\:\A\to\CAT$ be a $2$-functor with $\A$ small, and let $F\:{\left(\Grothop{W}\right)}\op\to \C$ be a $2$-functor. The \dfn{cartesian-marked lax conical colimit of $F$}, denoted as $\laxncolim{\Delta 1}{F}$, is (if it exists) an object $C\in \C$ together with a natural isomorphism of categories
		$$\HomC{\C}{C}{U}\iso \HomC{\mlaxn{\Grothopdiag{W}}{\CAT}}{\Delta 1}{\HomC{\C}{F(-)}{U}}$$
	\end{defne}
	
	\begin{rem}\label{fromgrothisnotrestrictive}
		Considering $2$-functors $F$ of the form $F\:\Grothop{W}\to\C$ in \defx\ref{laxnormalconicallimits} is not restrictive at all. Indeed any $2$-category $\B$ can be seen as the $2$-category of elements of the $2$-functor $\Delta 1\:\B\to\CAT$ constant at $\1$.
	\end{rem}
	
	\begin{rem}
		We now show a new, more elementary proof of the fact that cartesian-marked lax conical limits are particular weighted $2$-limits. Street states the analogous result in~\cite{street_limitsindexedbycatvalued} for all the general $2$-limits introduced by Gray in~\cite{gray_formalcattheory}, with a complex proof that gives the weight as the coidentifier of a certain $2$-cell with horizontal codomain the weight of lax conical limits. We present an original explicit weight for cartesian-marked lax conical limits that is actually simpler than the one for lax conical limits. Indeed the latter involves quotients of lax $2$-dimensional slices (see Street's~\cite{street_limitsindexedbycatvalued}), while the former only needs ordinary $1$-dimensional slices. The reason is that the laxness of cartesian-marked lax natural transformations is concentrated in the vertical part of $\Grothop{W}$.  
	\end{rem}
	
	\begin{teor}\label{laxnormalareweighted}
		Cartesian-marked lax conical limits are particular weighted $2$-limits. More precisely, given $2$-functors $Z\:\A\to\CAT$ and $F\:\Grothop{Z}\to\C$ with $\A$ small, the weight that realizes $\laxnlim{\Delta 1}{F}$ is
		\begin{fun}
			\Wlaxn & \: & \Grothop{Z} \hphantom{..}& \too & \CAT \\[1ex]
			&& \fib[n][3]{(B,X')}{(g,\beta)}{(C,X'')} &\mto & \fib[n][3]{\slice{Z(B)}{X'}}{\beta\c Z(g)(-)}{\slice{Z(C)}{X''}}\\[4.8ex]
			&& (g,\beta)\aR{\delta}(h,\gamma) & \mto & Z(\delta)_{\dom(-)}
		\end{fun}
		where the action of $\beta\c Z(g)(-)$ on morphisms\v is given by $Z(g)(\dom(-))$.
	\end{teor}
	\begin{proof}
		Given $\phi\:\Delta 1 \alaxn{}N$ a cartesian-marked lax natural transformation, we convert it into a $2$-natural transformation $[\phi]\:\Wlaxn\aR{} N$ setting, for every $(B,X')\in \Grothop{Z}$, 
		$$[\phi]_{(B,X')}(\id{X'})\deq \phi_{(B,X')}$$
		$$[\phi]_{(B,X')}\left(\begin{cd}*[1][0.5]
			X\arrow[rd,"{\alpha}"']\arrow[rr,"{\alpha}"] \&\& X'\arrow[ld,equal] \\
			\& X'
		\end{cd}\right)\deq \phi_{(\id{B},\alpha)}.$$
		$[\phi]$ extends in a unique way to a $2$-natural transformation.
	\end{proof}

	The following corollary will be useful for the proof of \thex\ref{redlaxnormalconical}.
	
	\begin{coroll}\label{laxnormalcolimareweighted}
		Cartesian-marked lax conical colimits are particular weighted $2$-colimits, and the weight that expresses them is $\Wlaxn$.
	\end{coroll}
	
		We now present our new proof of the fact that every weighted $2$-limit can be reduced to a cartesian-marked lax conical one. This was first proved by Street in~\cite{street_limitsindexedbycatvalued}; another proof can be derived from Proposition~3.18 of Szyld's~\cite{szyld_limlifttheortwodimmonad}. Our proof is based on the elementary \conx\ref{2SetGrothconstrtoconicalize}, allowing it to be understood by a wider audience.

	\begin{teor}\label{redlaxnormalconical}
		Every weighted $2$-limit can be reduced to a cartesian-marked lax conical one. More precisely, given $2$-functors $F\:\A\to \C$ and $W\:\A\to\CAT$ with $\A$ small,
		$$\wlim{W}{F}\iso \laxnlim{\Delta 1}{\left(F\c \groth{W}\right)}$$
		either side existing if the other does, where $\groth{W}$ is the $2$-category of elements of $W$.
	\end{teor}
	
	\begin{proof}
		By \remx\ref{needoflaxcones}, we can just essentially conicalize the weighted $2$-colimits $W\iso \wcolim{W}{\yy}$ with $W\:\A\to \CAT$ a $2$-copresheaf with $\A$ small. It is straightforward to extend \conx\ref{2SetGrothconstrtoconicalize} to an isomorphism of categories
		\begin{equation}\label{isoredlaxnormal}
			\scalebox{0.95}{$\HomC{\m{\A}{\CAT}}{W}{\HomC{\m{\A}{\CAT}}{\y{-}}{U}}\iso \HomC{\mlaxn{\Grothopdiag{W}}{\CAT}}{\Delta 1}{\HomC{\m{\A}{\CAT}}{{\left(\yy\c\h{\groth{W}}\op\right)}(-)}{U}}$}
		\end{equation}
		$2$-natural in $U\in\m{\A}{\CAT}$, expressing
		$$W\iso \wcolim{W}{\yy}\iso \laxncolim{\Delta 1}{\left(\yy\c\h{\groth{W}}\op\right)}.$$
		
		Consider now $2$-functors $F\:\A\to\C$ and $W\:\A\to\CAT$ with $\A$ small. Then by the argument above and \corx\ref{laxnormalcolimareweighted}
		$$W\iso \laxncolim{\Delta 1}{\left(\yy\c\h{\groth{W}}\op\right)}\iso \wcolim{\Wlaxn}{\left(\yy\c\h{\groth{W}}\op\right)}.$$
		By the lemma of continuity of a limit in its weight (see \remx\ref{needoflaxcones}) and \thex\ref{laxnormalareweighted},
		$$\wlim{W}{F}\iso\wlim{\Wlaxn}{\left(\wlim{\left(\yy\c\h {\groth{W}}\op\right)(-)}{F}\right)}\iso\wlim{\Wlaxn}{\left(F\c \groth{W}\right)}\iso\laxnlim{\Delta 1}{\left(F\c \groth{W}\right)}$$
		where the isomorphism in the middle is easy to prove.
	\end{proof}
	
	\begin{rem}\label{univlaxnormalcone}
		The proof of \thex\ref{redlaxnormalconical}, together with the proofs of \thex\ref{laxnormalareweighted} and \corx\ref{laxnormalcolimareweighted}, also shows how to obtain the correspondence between the universal cylinder of a weighted $2$-limit and the associated universal cartesian-marked lax cocone. Calling the two, respectively, 
		$$\lambda\:W\aR{} \HomC{\C}{L}{F(-)} \text{\quad and \quad}\hats{\lambda}\:\Delta 1 \alaxn{} \HomC{\C}{L}{\left(F\c \groth{W}\right)(=)}$$
		the correspondence is given, for every $(f,\alpha)\:(A,X)\to(B,X')$ in $\Grothop{W}$, by
		\begin{equation}\label{lambdahats}
			\hats{\lambda}_{(A,X)}=\lambda_A(X)\quad\text{ and }\quad\hats{\lambda}_{(f,\alpha)}=\lambda_B(\alpha).
		\end{equation}
	\end{rem}
	
	\begin{prop}\label{preservereflect}
		A weighted $2$-limit is preserved or reflected precisely when its associated cartesian-marked lax conical limit is so.
	\end{prop}
	\begin{proof}
		Clear after \remx\ref{univlaxnormalcone}.
	\end{proof}
	
	\begin{rem}\label{remoplaxnormalvslaxnormal}
		As weighted $2$-colimits in $\C$ are weighted $2$-limits in $\C\op$, we automatically obtain from \thex\ref{redlaxnormalconical} the reduction of weighted $2$-colimits to cartesian-marked lax conical ones. However, this considers the 2-category of elements $\groth{W}\:\Grothop{W}\to\A\op$ of the weight $W\:\A\op\to \Cat$, rather than the more natural $\h\groth{W}\:\Groth{W}\to\A$ with $\Groth{W}$ defined as follows:
		\begin{description}
			\item[an object of $\Groth{W}$] is a pair $(A,X)$ with $A\in \A$ and $X\in F(A)$;
			\item[a morphism $(A,X)\to (B,X')$ in $\Groth{W}$] is a pair $(f,\alpha)$ with $f\:A\to B$ a morphism in $\A$ and $\alpha\:X\to W(f)(X')$ a morphism in $W(A)$;
			\item[a $2$-cell $(f,\alpha)\aR{}(g,\beta)\:(A,X)\to (B,X')$ in $\Groth{W}$] is a $2$-cell $\delta\:f\aR{}g$ in $\A$ such that $W(\delta)_{X'}\c \alpha=\beta$.
		\end{description}
		We believe that it is more natural to reduce weighted $2$-limits to cartesian-marked lax conical ones and weighted $2$-colimits to \predfn{cartesian-marked oplax conical} ones. This idea is original and brings to \thex\ref{oplaxnormalareweighted} and \thex\ref{redoplaxnormalconical}.
	\end{rem}
	
	\begin{defne}
		Let $W\:\A\op\to \CAT$ be a $2$-functor with $\A$ small, and consider $2$-functors $M,N\:{\left(\Groth{W}\right)}\op\to \D$. A \dfn{cartesian-marked oplax natural transformation $\alpha$ from $M$ to $N$}, denoted $\alpha\:M\aoplaxn{}N$,\v is an oplax natural transformation $\alpha$ from $M$ to $N$ such that the structure $2$-cell on every morphism $\p{f,\id{}}$ in ${\left(\Groth{W}\right)}\op$ is the identity.
	\end{defne}
	
	\begin{defne}\label{defoplaxnormalconical}
		Consider 2-functors $W\:\A\op\to \CAT$ and $F\:\Groth{W}\to \C$ with $\A$ small. The \dfn{cartesian-marked oplax conical $2$-colimit of $F$}, denoted as $\oplaxncolim{F}$, is (if it exists) an object $C\in \C$ together with a $2$-natural isomorphism of categories
		$$\HomC{\C}{C}{U}\iso \HomC{\moplaxn{{\left(\Grothdiag{W}\right)}\op}{\CAT}}{\Delta 1}{\HomC{\C}{F(-)}{U}}$$
		
		\noindent When $\oplaxncolim{F}$ exists, the identity on $C$ provides the \dfn{universal cartesian-marked oplax cocone} $\mu\:\Delta 1 \aoplaxn{}\HomC{\C}{F(-)}{C}$.
	\end{defne}
	
	\begin{teor}\label{oplaxnormalareweighted}
		Cartesian-marked oplax conical $2$-colimits are particular weighted $2$-colimits. More precisely, given $2$-functors $Z\:\A\op\to\CAT$ and $F\:\Groth{Z}\to\C$ with $\A$ small, the weight that realizes $\oplaxncolim{F}$ is
		\begin{fun}
			\Woplaxn & \: & {\left(\Groth{Z}\right)}\op \hphantom{.}& \too & \CAT \\[1ex]
			&& (B,X') &\mto & \coslice{X'}{Z(B)}
		\end{fun}
	\end{teor}
	\begin{proof}
		The proof is analogous to the one of \thex\ref{laxnormalareweighted}.
	\end{proof}
	
	\begin{teor}\label{redoplaxnormalconical}
		Every weighted $2$-colimit can be reduced to a cartesian-marked oplax conical one. Given $2$-functors $F\:\A\to \C$ and $W\:\A\op\to\CAT$ with $\A$ small,
		$$\wcolim{W}{F}\iso \oplaxncolim{\left(F\c \groth{W}\right)}$$
		where $\groth{W}\:\Groth{W}\to\A$ is the $2$-category of elements of $W$.
	\end{teor}
	\begin{proof}
		The proof is analogous to the one of \thex\ref{redlaxnormalconical}.
	\end{proof}
	
	\begin{rem}\label{univoplaxnormalcocone}
		We then also obtain analogues of \remx\ref{univlaxnormalcone} (with the same formulas) and of \prox\ref{preservereflect}.
	\end{rem}
	
	\begin{exampl}\label{univoplaxnormalcoconepresheaves}
		By the proof of \thex\ref{redoplaxnormalconical}, for every $W\:\A\op\to\CAT$ with $\A$ small
		$$W\iso \oplaxncolim{\left(\yy\c\h{\groth{W}}\right)}.$$
		The universal cartesian-marked oplax cocone is given by
		$$\forall \fib[o][2.3]{(B,X')}{(f,\alpha)}{(A,X)} \text{ in } {\Grothdiag{W}} \quad\quad\quad{
			\begin{cd}*[3.4][1.8]
				\y{A} \arrow[rr,"{\ceil{X}}",""'{name=Q}] \arrow[d,"{\y{f}}"'] \&\& W\\
				\y{B} \arrow[rru,bend right,"{\ceil{X'}}"']\& \phantom{.} \arrow[Rightarrow,from=Q,"{\ceil{\alpha}}"{pos=0.36},shift right=1.5ex,shorten <=1ex,shorten >= 2.7ex]
		\end{cd}}$$
		
		In particular, taking $\A=\1$, $W$ is a small category $\D$ and $\groth{W}$ is ${\D}\to{\1}$. We obtain that $\1$ is ``cartesian-marked oplax conical dense", building $\D$ with universal cocone
		$$\forall \fib[o]{D}{f}{C} \text{ in } \D \quad\quad\quad {\begin{cd}*[3.4][2]
				\1 \arrow[rr,"C",""'{name=Q}] \arrow[d,equal] \&\& \D\\
				\1 \arrow[rru,bend right,"D"']\& \phantom{.} \arrow[Rightarrow,from=Q,"{f}"{pos=0.355},shift right=1.5ex,shorten <=0.8ex,shorten >= 2.45ex]
		\end{cd}}$$
	\end{exampl}

	\section{Pointwise Kan extensions along discrete 2-fibrations}\label{sectionkanextensions}
	
	In this section, we propose an original definition of pointwise Kan extension in $\twoCATlax$ along a discrete $2$-opfibration. Our motivating application is a $2$-dimensional generalization of the fact that the category of elements can be abstractly captured by a comma object that also exhibits a pointwise Kan extension. We will prove it in \thex\ref{teorpointkanextforgroth} using such definition. We explain why we should consider $\twoCATlax$ in order to prove this, and then we inscribe $\twoCATlax$ in an original idea of $2$-$\V$-enrichment. It is the concept of $2$-$\V$-enriched category that guides us to a notion of colimit in a $2$-$\Set$-enriched category and then to our notion of pointwise Kan extension in $\twoCATlax$.
	
	Our motivations are to capture the Grothendieck construction in an abstract way, towards an enriched version of the Grothendieck construction, and to understand how the various properties of the $2$-category of elements are connected with each other. We will show in Section~\ref{sectiontwosetgrothconstr} that the pointwise Kan extension result which we prove in~\thex\ref{teorpointkanextforgroth} implies many other properties. 
	
	An important ingredient will be that a pointwise Kan extension in $\twoCATlax$ as originally defined in \defx\ref{defpointkanextensionlaxthreecat} is always a weak Kan extension (\defx\ref{defweakkanextensionlaxthreecat}) as well. The proof, in \prox\ref{pointkanisalsoweak}, will be based on an $\oplaxn\mbox{-}\lax$ generalization of the parametrized Yoneda lemma, proved in \thex\ref{parametrizedyonedaoplaxnlax}, that does not seem to appear in the literature.
	
	The first problem that we encounter is to understand which ambient hosts the $2$-category of elements construction. Aiming at a $2$-dimensional generalization of the fact that the category of elements can be captured by a comma object that also exhibits a pointwise left Kan extension, we observe the following proposition, due to Bird.
	
	\begin{prop}[Bird~\cite{bird_limitsintwocatoflocpresented}]\label{laxnatinsidegrothconstr}
		Let $F\:\B\to \CAT$ be a $2$-functor and consider its $2$-category of elements. There is a cartesian-marked lax natural transformation $\lambda$ of the form
		\begin{eqD*}
			\sq*[l][5][5][\laxn][2.7][2.2][0.52]{\Grothop{F}}{\1}{\B}{\CAT}{}{\groth{F}}{\1}{F}
		\end{eqD*}
	\end{prop}
	\begin{proof}
		Given a morphism $(f,\alpha)\:(A,X)\to (B,X')$ in $\Grothop{F}$, $\lambda_{(A,X)}$ is the functor $\1\to F(A)$ corresponding to $X\in F(A)$ and $\lambda_{(f,\alpha)}$ corresponds to $\alpha$. 
	\end{proof}
	
	\begin{rem}\label{needof2catlax}
		\prox\ref{laxnatinsidegrothconstr} forces us to move out of $\twoCAT$ in order to capture the $2$-category of elements (but also just the usual Grothendieck construction) from an abstract point of view. Indeed, we need to at least admit lax natural transformations as $2$-cells. If we wish to recover the Grothendieck construction of pseudofunctors or of general lax functors into $\CAT$, we also need to admit lax functors as $1$-cells of our ambient. We will just consider strict $2$-functors for simplicity, but we actually expect everything to hold for lax functors as well.
		
		We thus consider the lax $3$-category $\twoCATlax$ of $2$-categories, $2$-functors, lax natural transformations and modifications. In~\cite{lambert_discretetwofib}, Lambert has indeed proved that this forms a lax $3$-category, i.e.\ a category enriched over the cartesian closed $1$-category of $2$-categories and normal lax functors. The idea to use $\twoCATlax$ is reinforced by the fact that Buckley, in~\cite{buckley_fibredtwocatandbicat} (continuing the work of Bakovi\'c's~\cite{bakovic_fibrofbicat}), found the need to consider trihomomorphisms $F\:\B\coop\to \twoCATlax$ in order to capture non-split Hermida's $2$-fibrations (\cite{hermida_somepropoffibasafibredtwocat}) via a suitable Grothendieck construction.
		
		We should keep in mind that $\twoCATlax$ has no underlying $2$-category, since the interchange rule now only holds in a lax version, in the sense that we have a modification between the two possible lax natural transformations. Indeed, consider two lax natural transformations
		\begin{cd}
			\A\arrow[r,bend left,"{F}",""'{name=A}]\arrow[r,bend right,"{G}"',""{name=B}]\&\B \arrow[r,bend left,"{H}",""'{name=C}]\arrow[r,bend right,"{K}"',""{name=D}]\&\C;
			\arrow[from=A,to=B,Rightarrow,"{\alpha}"]
			\arrow[from=C,to=D,Rightarrow,"{\beta}"]
		\end{cd}
		Then for every $A\in \A$, the component $\alpha_A$ is a morphism $F(A)\ar{}G(A)$ in $\B$ and we can consider the structure $2$-cell $\beta_{\alpha_A}$ of $\beta$ on such morphism
		\begin{cd}[4.75][4.75]
			H(F(A))\arrow[r,"{\beta_{F(A)}}"]\arrow[d,"{H(\alpha_A)}"']\& K(F(A))\arrow[d,"{K(\alpha_A)}"]\arrow[ld,Rightarrow,"{\beta_{\alpha_{A}}}",shorten <=2.7ex,shorten >=2.2ex]\\
			H(G(A))\arrow[r,"{\beta_{G(A)}}"']\& K(G(A))
		\end{cd}
		The $\beta_{\alpha_A}$'s collect into a modification $\beta_\alpha$ by lax naturality of $\beta$.
	\end{rem}
		
	\begin{rem}\label{remarchetypalthreedimclassifprocchanged}
		 To reach the right notion of pointwise Kan extension in $\twoCATlax$, it will be helpful to inscribe $\twoCATlax$ in an original idea of $2$-$\V$-enrichment. Consider the following chain:
		$$\Set\quad\leadsto\quad \CAT \quad\leadsto\quad \twoCATlax$$
		Such a chain is motivated by Weber's~\cite{weber_yonfromtwotop}, that explains how the 2-dimensional classifier of $\Cat$, which is the archetypal elementary 2-topos, is given by the category of elements. Analogously, we can conceive the subobject classifier of $\Set$, which is the archetypal elementary topos, as the 0-category of elements. We will say in \remx\ref{remapplelementarythreetopos} that we can think of $\twoCATlax$, which hosts the 2-category of elements, as the archetypal 3-dimensional elementary topos. We believe that the sequence above is best captured by what we call a \predfn{$2$-$\V$-enrichment} with $\V=\Set$:
		$$\V\quad\leadsto\quad \VCAT{\V} \quad\leadsto\quad \twoVCAT{\V}$$
		Enriching over $\VCAT{\V}$, we should take into account the fact that $\VCAT{\V}$ is a monoidal $2$-category, and then take a \predfn{weak enrichment} rather than an ordinary one.

		The concept of \predfn{weak enrichment} is explored in Garner and Shulman's~\cite{garnershulman_enrichedcatasafreecocompl}, but the equivalent formulation presented below does not seem to appear in the literature.
	\end{rem}
	
	\begin{cons}[Weak enrichment]\label{consweakenrichment}
		Recall that if $(\V,\otimes,I)$ is a monoidal category with coproducts such that $-\otimes -$ preserves coproducts in each variable, we can define a $\V$-enriched category as pair $(S,\A)$ with $S$ a set (of objects) and $\A$ a monoid in the monoidal category $\m{S\times S}{\V}$ of square $S$-indexed matrices with entries in $\V$, with the tensor product given by matrix multiplication. $\V$-enriched functors can be defined on this line as well.

		Let $(\K,\otimes,I,\alpha,\lambda,\rho)$ be a monoidal $2$-category, i.e.\ a $2$-category $\K$ that is monoidal in the $1$-dimensional sense but such that the tensor product is a $2$-functor $\K\times \K\to \K$. And assume that $\K$ has coproducts and that $-\otimes -$ preserves them in each variable. Then, for every set $S$, the $2$-category $\m{S\times S}{\K}$ is $2$-monoidal as well, with tensor product given by matrix multiplication.
		
		We define a \dfn{$\K$-weakly enriched category} as a pair $(S,\A)$ with $S$ a set (of objects) and $\A$ a pseudomonoid in the monoidal $2$-category $\m{S\times S}{\K}$ of square $S$-indexed matrices with entries in $\K$. Notice that a strict $2$-monoid in $\m{S\times S}{\K}$ is a ${\K}_{\h[3]0}$-enriched category with object set $S$.

		Given two $\K$-weakly enriched categories $(S,\A)$ and $(T,\B)$, we define a \dfn{$\K$-weakly enriched functor $(S,\A)\to (T,\B)$} as a pair $(F,\overline{F})$ with $F\:S\to T$ a function and $\overline{F}\:\A\to F\st\B$ a lax morphism between lax monoids in $\m{S\times S}{\K}$, where 
		$$F\st \B =\left(S\times S\ar{F\times F} T\times T \ar{\B}\V\right).$$

		Given now $(F,\overline{F}), (G,\overline{G})\:(S,\A)\to (T,\B)$ two $\K$-weakly enriched functors, we define a \dfn{$\K$-weakly enriched natural transformation $\phi\:(F,\overline{F})\alax{}(G,\overline{G})$} as a collection of $1$-cells
		$$\phi_A\:I\to \HomC{\B}{F(A)}{G(A)}$$
		in $\K$ for every $A\in S$ and $2$-cells
		\begin{cd}[-0.35][3.8]
			\& \A(A,B) \otimes I \arrow[r, "G\otimes \phi_{A}",""'{name=A}] \&[1.9ex] \B(GA,GB) \otimes \B(FA,GA) \arrow[rd, "\opn{comp}"{inner sep=0.3ex}] \\
			\A(A,B) \arrow[rd, "\lambda_{\A(A,B)}^{-1}"'{inner sep=0.3ex, pos=0.7}] \arrow[ru, "\rho_{\A(A,B)}^{-1}"{pos=0.7}] \&\&\&[-3ex] \B(FA,GB) \\
			\& I \otimes \A(A,B) \arrow[r, "\phi_{B}\otimes F"',""{name=B}] \& \B(FB,GB) \otimes \B(FA,FB) \arrow[ru, "\opn{comp}"'{pos=0.45}]
			\arrow[Rightarrow,from=A,to=B,shift left=2.85ex,"{\h[1]\phi_{A,B}}",shorten <=1.5ex,shorten >=1.5ex]
		\end{cd}
		in $\K$ for every pair $(A,B)\in S\times S$ that satisfy appropriate axioms, see Garner and Shulman's~\cite{garnershulman_enrichedcatasafreecocompl}. A notion of $\K$-weakly-enriched modification can be given as well.
	\end{cons}
	
		We notice that \conx\ref{consweakenrichment} is particularly useful when $\K$ is $\VCAT{\V}$, since we obtain a construction that we can apply to a starting ordinary $\V$.
	
	\begin{defne}\label{deftwovenrichment}
		Let $\V$ be a nice enough monoidal category so that $\VCAT{\V}$ with the tensor product of $\V$-categories becomes a monoidal $2$-category with coproducts such that its tensor product preserves them in each variable. We call \dfn{$2$-$\V$-enrichment} the $\VCAT{\V}$-weak enrichment.
	\end{defne}
	
	\begin{exampl}\label{exampletwosetenrichment}
		Consider $\V=\Set$. Then $2$-$\Set$-enriched categories, functors, natural transformations and modifications are, respectively, bicategories, lax functors, lax natural transformations and modifications.
		
		We will think of $\twoCATlax$ as the lax $3$-category of $2$-$\Set$-enriched categories, restricted to strict weakly enriched categories and functors for simplicity.
	\end{exampl}
	
	\begin{rem}\label{remnametwosetenrichedgroth}
		We can now think of the $2$-category of elements construction as the \predfn{$2$-$\Set$-enriched Grothendieck construction}. It takes enriched presheaves with values in $\VCAT{\Set}=\CAT$, that is the base of the weak enrichment we consider, and converts them into discrete $2$-opfibrations, that can be seen as \predfn{$2$-$\Set$-opfibrations}.
		
		In future work, we will explore how we can generalize this to more general monoidal categories $\V$ in the place of $\Set$, thus obtaining an enriched Grothendieck construction. In particular, we believe that for this the $2$-$\Set$-enrichment should be distinguished from the ordinary $\Cat$-enrichment; we actually expect the latter to be less useful.
	\end{rem}
	
	We give an original definition of weak Kan extension in a lax $3$-category.
	
	\begin{defne}\label{defweakkanextensionlaxthreecat}
		A diagram
		\begin{cd}[5][4]
			\B \arrow[rr,"{F}"]\arrow[d,"{K}"']\&\arrow[dl,Rightarrow,shift right=0.75ex,"{\lambda}"{pos=0.355, inner sep= 0.35ex},shorten <=1.45ex,shorten >= 3ex]\& \C \\
			\A\arrow[rru,"L"']
		\end{cd}
		in a lax $3$-category $\Q$ (that is a category enriched over the $1$-category of $2$-categories and lax functors), exhibits $L$ as the \dfn{weak left Kan extension of $F$ along $K$}, written $L={\lan{K}{F}}$, if pasting with $\lambda$ gives an isomorphism of categories
		\begin{equation}\label{eqweakkanextensionlaxthreecat}
			\HomC{\HomC{\Q}{\A}{\C}}{L}{U}\iso \HomC{\HomC{\Q}{\B}{\C}}{F}{U\c K}
		\end{equation}
		for every $U\in \HomC{\Q}{\A}{\C}$ (the $2$-naturality in $U$ is granted automatically).
	\end{defne}
	
	\begin{rem}\label{remhomtwocategoriestwocatlax}
		Lambert showed in~\cite{lambert_discretetwofib} that $\twoCATlax$ is a lax $3$-category with hom-$2$-categories $$\HomC{\twoCATlax}{\A}{\C}\deq \mlax{\A}{\C}$$
		the $2$-category of $2$-functors from $\A$ to $\C$, lax natural transformations and modifications.
		
		So, in the notation of \defx\ref{defweakkanextensionlaxthreecat}, the isomorphism of categories of equation~\refs{eqweakkanextensionlaxthreecat} becomes, for every $U\in \mlax{\A}{\C}$,
		$$\HomC{\mlax{\A}{\C}}{L}{U}\iso \HomC{\mlax{\B}{\C}}{F}{U\c K}.$$
	\end{rem}
	
	\begin{rem}\label{remaimatpointwise}
		We aim at a notion of \predfn{pointwise Kan extension in $\twoCATlax$} that allows us to prove (in \thex\ref{teorpointkanextforgroth}) a $2$-dimensional version of the pointwise Kan extension result that we have for the category of elements. For this we need a notion of (co)limit in the setting of $\twoCATlax$. There surely is a natural notion of ``external limit" in this context, i.e.\ a limit in the whole $\twoCATlax$, which is that of limit enriched over the $1$-category of $2$-categories and lax functors. Indeed $\twoCATlax$ is enriched over such $1$-category. For $\C$ a $2$-category, this recovers e.g.\ the natural isomorphism
		$$\HomC{\twoCATlax}{\U}{\laxcomma{\Id{\C}}{\Id{\C}}}\iso \HomC{\twoCATlax}{\2}{\HomC{\twoCATlax}{\U}{\C}}$$
		presented by Lambert in~\cite{lambert_discretetwofib} as the universal property of the lax comma object $\laxcomma{\Id{\C}}{\Id{\C}}$ (see \defx\ref{completeunivproplaxcomma}) being the power of $\C$ by $\2$ in $\twoCATlax$.
		
		But we are more interested in an ``internal" notion of colimit, i.e.\ a notion of colimit in a $2$-$\Set$-enriched category (after \exax\ref{exampletwosetenrichment}). This should include the notion of colimit in a $2$-category, but be able to express the laxness as well. After Section~\ref{sectionlaxnormallimits}, we use (now non-necessarily conical) cartesian-marked oplax colimits as colimits in a $2$-$\Set$-category (see \defx\ref{defoplaxnormaltwocolimit}). Notice that the marking is a piece of structure. These colimits are a particular case of the sigma-omega-limits (or sigma-s-limits) considered by Szyld in~\cite{szyld_limlifttheortwodimmonad} and~\cite{szyld_liftingpielimitswithstrictproj}; they are a stricter version of Descotte, Dubuc and Szyld's~\cite{descottedubucszyld_sigmalimandflatpseudofun} weighted sigma limits. By Proposition~3.18 of Szyld's~\cite{szyld_limlifttheortwodimmonad}, these colimits can all be reduced to (still marked) conical ones. However, expressing the diagram, the weight and the marking as different pieces of structure makes it easier for us to define \predfn{pointwise Kan extensions in $\twoCATlax$}.
	\end{rem}
	
	\begin{defne}\label{defoplaxnormaltwocolimit}
		Let $M\:\A\op\to \CAT$ (the marking), $F\:\Groth{M}\to \C$ (the diagram) and $W\:{\left(\Groth{M}\right)}\op\to \CAT$ (the weight) be $2$-functors with $\A$ small. The \dfn{cartesian oplax colimit of $F$ marked by $M$ and weighted by $W$}, denoted as $\oplaxnmcolim{M}{W}{F}$, is (if it exists) an object $C\in \C$ together with an isomorphism of categories
		$$\HomC{\C}{C}{U}\iso \HomC{\moplaxn{{\left(\Grothdiag{M}\right)}\op}{\CAT}}{W}{\HomC{\C}{F(-)}{U}}$$
		$2$-natural in $U\in \C$. When $\oplaxnmcolim{M}{W}{F}$ exists, the identity on $C$ provides a cartesian-marked oplax natural transformation $\mu\:W \aoplaxn{}\HomC{\C}{F(-)}{C}$ called the \dfn{universal cartesian-marked oplax cocylinder}.
		
		We will also need to consider the case in which the domain of $F$ is expressed as $\Grothop{M}$ for some $2$-functor $M\:\A\to \CAT$, and $W\:{\left(\Grothop{M}\right)}\op\to \CAT$. The \dfn{cartesian oplax colimit of $F$ opmarked by $M$ and weighted by $W$}, denoted as $\oplaxnopmcolim{M}{W}{F}$, is (if it exists) an object $C\in \C$ together with a $2$-natural isomorphism of categories
		$$\HomC{\C}{C}{U}\iso \HomC{\moplaxn{{\left(\Grothopdiag{M}\right)}\op}{\CAT}}{W}{\HomC{\C}{F(-)}{U}}$$
	\end{defne}
	
	\begin{rem}
		Cartesian-marked oplax with respect to the marking described in \remx\ref{fromgrothisnotrestrictive}, which we can call \dfn{the trivial marking} and denote as $\opn{triv}$, coincides with strict $2$-naturality. We can then rephrase \thex\ref{redoplaxnormalconical} as follows:

		\noindent \textit{In $\twoCATlax$ every trivially-marked weighted $2$-colimit can be equivalently expressed as a marked trivially-weighted $2$-colimit. More precisely}
		$$\oplaxnmcolim{\opn{triv}}{W}{F}\h[2]\iso\h[2] \oplaxnmcolim{W}{\Delta 1}{\left(F\c \groth{W}\right)}$$
	\end{rem}
	
	\begin{exampl}
		Let $F\:\A\op\to \CAT$ be a $2$-functor with $\A$ small. Then by \exax\ref{univoplaxnormalcoconepresheaves}
		$$F\iso \oplaxnmcolim{F}{\Delta 1}{\left(\yy\c\h \groth{F}\right)}$$
		
		In particular, taking $\A=1$, we obtain that for every small category $\D$
		$$\D\iso \oplaxnmcolim{\D}{\Delta 1}{\Delta 1}.$$
		Notice that the marking given by $\D$ is ``chaotic", in the sense that cartesian-marked oplax with respect to it coincides with oplax. In $\twoCATlax$ the $2$-functor $\1\:\1\to \CAT$ is thus conically dense, analogously to the fact that the functor $\1\:\1\to \Set$ is dense.
	\end{exampl}
	
	\begin{rem}
		We now propose an original notion of pointwise left Kan extension in $\twoCATlax$ along a discrete $2$-opfibration, using our definition of colimit in such setting (\defx\ref{defoplaxnormaltwocolimit}). Our idea is to keep the corresponding diagram and weight considered in the (ordinarily) enriched setting (see Dubuc's~\cite{dubuc_kanextensionsinenrichedcat} or Kelly's~\cite{kelly_basicconceptsofenriched} for the classical definition), but adding the marking that we naturally have when we extend along a discrete $2$-opfibration. In \thex\ref{teorpointkanextforgroth}, this notion will allow us to prove a $2$-dimensional version of the fact that the comma object that captures the category of elements exhibits a pointwise Kan extension.
	\end{rem}
	
	\begin{defne}\label{defpointkanextensionlaxthreecat}
		Consider a diagram
		\begin{cd}[5][4]
			\B \arrow[rr,"{F}"]\arrow[d,"{K}"']\&\arrow[dl,Rightarrow,shift right=0.75ex,"{\lambda}"{pos=0.355, inner sep= 0.35ex},shorten <=1.45ex,shorten >= 3ex]\& \C \\
			\A\arrow[rru,"L"']
		\end{cd}
		in $\twoCATlax$ with $\B$ small and $K$ a discrete $2$-opfibration with small fibres. Then by \thex\ref{essentialimagegivenbytwosetopf}, $K$ is isomorphic in the slice $\slice{\twoCAT}{\A}$ to $\groth{M}$ for some $2$-functor $M\:\A\to \CAT$. We can assume $K$ is in the form $\groth{M}$, up to whiskering the diagram with the isomorphism in the slice. Assume further that $\lambda$ is a cartesian-marked lax natural transformation with respect to $M$.
		
		We say that $\lambda$ exhibits $L$ as the \dfn{pointwise left Kan extension of $F$ along $K$}, written $L=\Lan{K}{F}$, if for every $A\in \A$
		$$L(A)\iso \oplaxnopmcolim{M}{\HomC{\A}{K(-)}{A}}{F}$$
		with universal cartesian-marked oplax cocylinder
		\begin{equation}\label{eqpointkanextensiontwocatlax}
			\HomC{\A}{K(-)}{A}\aRR{L}\HomC{\C}{(L\c K)(-)}{L(A)}\aRR[\oplaxn]{\HomC{\C}{\lambda_{-}}{\id{}}}\HomC{\C}{F(-)}{L(A)};
		\end{equation}
		or equivalently if for every $A\in \A$ and every $C\in \C$ the functor  
		$$\HomC{\C}{L(A)}{C}\arr{}\HomC{\moplaxn{\B\op}{\CAT}}{\HomC{\A}{K(-)}{A}}{\HomC{\C}{F(-)}{C}}$$
		given by the cartesian-marked oplax natural transformation of equation~\refs{eqpointkanextensiontwocatlax} is an isomorphism of categories (notice that the $2$-naturality in $C$ and $A$ is granted, where the latter is using that $L$ is a $2$-functor).
	\end{defne}
	
	\begin{rem}
		The rest of this section is dedicated to the proof that every pointwise left Kan extension in $\twoCATlax$ along a discrete $2$-opfibration (as defined in \defx\ref{defpointkanextensionlaxthreecat}) is a weak left Kan extension in $\twoCATlax$ as well.
		
		For this, we need a generalization of the parametrized Yoneda lemma which is cartesian-marked oplax and lax together (\thex\ref{parametrizedyonedaoplaxnlax}). Such result does not seem to appear in the literature. While a fully lax parametrized Yoneda lemma is not possible, since it is the strict naturality that classically allow to expand the datum on the identity to a complete natural transformation, our version shows the minimal strictness needed to do so.
		
		Interestingly, in the fully strict $2$-natural case such expansion depends on the naturality of what will be our parameter $A$. Instead, we will need to expand via the slight strictness of the cartesian-marked oplax naturality in $B$, and this is harder to achieve.
	\end{rem}
	
	\begin{defne}\label{defoplaxnlax}
		Let $G,H\:\B\op\times \C\to \E$ be $2$-functors. Assume that $\B$ is endowed with the cartesian marking presented by a (split) discrete 2-opfibration $K\:\B\to \A$ (with small fibres). That is, $\B\iso \Grothop{M}$ for some 2-functor $M\:\A\to \Cat$ and has the cartesian marking. An \dfn{$\text{\textbf{oplax}}^\text{\textbf{cart}}\mbox{-}\text{\textbf{lax}}$ natural transformation} $\alpha$ from $G$ to $H$ is a collection of morphisms
		$$\alpha_{B,C}\:G(B,C)\to H(B,C)$$
		in $\E$ for every $(B,C)\in \B\op\times\C$ and, for every $f\:B'\to B$ in $\B\op$ and $g\:C\to C'$ in $\C$, structure $2$-cells
		\begin{eqD*}
			\sq*[o][4.8][4.8][\alpha_{f,C}][3.1][3]{G(B',C)}{H(B',C)}{G(B,C)}{H(B,C)}{\alpha_{B',C}}{G(f,\id{})}{H(f,\id{})}{\alpha_{B,C}}\qquad
			\sq*[l][4.8][4.8][\alpha_{B,g}][3.1][3]{G(B,C)}{H(B,C)}{G(B,C')}{H(B,C')}{\alpha_{B,C}}{G(\id{},g)}{H(\id{},g)}{\alpha_{B,C'}}
		\end{eqD*}
		such that $\alpha_{-,C}$ is cartesian-marked oplax natural in $B\in \B\op$ and $\alpha_{B,-}$ is lax natural in $C\in \C$, and such that the following compatibility axiom holds:\v[0.5]
		\begin{eqD*}
			\begin{cd}*[4.4][-3.4]
				\&[-1ex] G(B,C)\arrow[rr,"{\alpha_{B,C}}"]\arrow[rd,"{G(\id{},g)}"{inner sep=0.2ex},""'{name=A}]\&[-1ex]\& H(B,C)\arrow[rd,"{H(\id{},g)}"{inner sep=0.2ex}]\arrow[d,Rightarrow,shift right=2.85ex,"{\alpha_{B,g}}"{pos=0.44},shorten <=0.8ex,shorten >=2.2ex]\\
				G(B',C)\arrow[rd,"{G(\id{},g)}"'{inner sep=0.2ex},""{name=B}]\arrow[ru,"{G(f,\id{})}"{inner sep=0.2ex}]\&\&G(B,C')\arrow[rr,"{\alpha_{B,C'}}"]\&\vphantom{.}\arrow[d,Rightarrow,shift right=2.85ex,"{\alpha_{f,C'}}",shorten <=1.05ex,shorten >=1ex]\& H(B,C')\\
				\& G(B',C')\arrow[ru,"{G(f,\id{})}"'{inner sep=0.2ex}]\arrow[rr,"{\alpha_{B',C'}}"']\&\&H(B',C')\arrow[ru,"{H(f,\id{})}"'{inner sep=0.2ex}]
				\arrow[equal,from=A,to=B,shorten <=3.6ex,shorten >=3.6ex]
			\end{cd}\h[2] = \h[2] 
			\begin{cd}*[4.4][-3.4]
				\& G(B,C)\arrow[rr,"{\alpha_{B,C}}"]\arrow[d,Rightarrow,shift left=2.85ex,"{\alpha_{f,C}}"'{pos=0.44},shorten <=0.8ex,shorten >=2.2ex]\&\&[-1ex] H(B,C)\arrow[rd,"{H(\id{},g)}"{inner sep=0.2ex},""'{name=A}]\\
				G(B',C)\arrow[rr,"{\alpha_{B',C}}"]\arrow[rd,"{G(\id{},g)}"'{inner sep=0.2ex}]\arrow[ru,"{G(f,\id{})}"{inner sep=0.2ex}]\&\vphantom{.}\arrow[d,Rightarrow,shift left=2.85ex,"{\alpha_{B',g}}"',shorten <=1.05ex,shorten >=1ex]\&H(B',C)\arrow[ru,"{H(f,\id{})}"{inner sep=0.2ex}]\arrow[rd,"{H(\id{},g)}"'{inner sep=0.2ex},""{name=B}]\&\&[-1ex] H(B,C')\\
				\& G(B',C')\arrow[rr,"{\alpha_{B',C'}}"']\&\&H(B',C')\v[2]\arrow[ru,"{H(f,\id{})}"'{inner sep=0.2ex}]
				\arrow[equal,from=A,to=B,shorten <=3.6ex,shorten >=3.6ex]
			\end{cd}
		\end{eqD*}
		
		A \dfn{modification} $\Theta\:\alpha\aM{}\beta\:G\aR[\opln\mbox{-}\lax]{}H$ between $\oplaxn\mbox{-}\lax$ natural transformations is a collection of $2$-cells 
		\tc+[4]{G(B,C)}{H(B,C)}{\alpha_{B,C}}{\beta_{B,C}}{\Theta_{B,C}}
		in $\E$ that forms both, fixing $C$, a modification $\alpha_{-,C}\aM{}\beta_{-,C}$, and, fixing $B$, a modification $\alpha_{B,-}\aM{}\beta_{B,-}$.
	\end{defne}
	
	\begin{teor}[The $\oplaxn\mbox{-}\lax$ parametrized Yoneda lemma]\label{parametrizedyonedaoplaxnlax}
		Let $K\:\B\to \A$ be a \pteor{split} discrete $2$-opfibration \pteor{with small fibres} and $F\:\B\op\times \A\to \CAT$ be a $2$-functor. There is a bijection between
		$$\alpha_{B,A}\:\HomC{\A}{K(B)}{A}\to F(B,A)$$
		$\oplaxn\mbox{-}\lax$ natural in $(B,A)\in \B\op\times \A$ and
		$$\eta_B\:1\to F(B,K(B))$$
		extraordinary lax natural in $B\in \B$ \pteor{see Hirata's~\cite{hirata_notesonlaxends} for a definition of extraordinary lax natural transformations and modifications between them}.
		
		Moreover this bijection extends to an isomorphism of categories, considering as morphisms of the two categories respectively the modifications between $\oplaxn\mbox{-}\lax$ natural transformations and the modifications between extraordinary lax natural transformations.
	\end{teor}
	\begin{proof}
		By \thex\ref{essentialimagegivenbytwosetopf}, we can assume that $K$ is in the form $\groth{M}\:\Grothop{M}\to \A$ for a $2$-functor $M\:\A \to \CAT$. Given
		$$\alpha_{(B,Y),A}\:\HomC{\A}{K(B,Y)}{A}\to F((B,Y),A)$$
		$\oplaxn\mbox{-}\lax$ natural in $((B,Y),A)\in {\left(\Grothop{M}\right)}\op\times \A$ (see \defx\ref{defoplaxnlax}), we construct $\eta_{(B,Y)}$ as the composite
		$$1\ar{\id{B}} \HomC{\A}{B}{B}\ar{\alpha_{(B,Y),B}}F((B,Y),B)$$
		Then $\eta_{(B,Y)}$ is extraordinary lax natural in $(B,Y)\in \Grothop{M}$, with structure $2$-cell on $(g,\gamma)\:(B,Y)\to (B',Y')$ in $\Grothop{M}$ given by the pasting
		\begin{eqD}{diagramfromalphatoeta}
			\begin{cd}*[4.5][4.5]
				1 \arrow[r,"{\id{B}}"]\arrow[d,"{\id{B'}}"']\& \HomC{\A}{B}{B} \arrow[ld,equal,shorten <=4.75ex,shorten >=4.75ex]\arrow[d,"{g\c -}"] \arrow[r,"{\alpha_{(B,Y),B}}"]\&F((B,Y),B) \arrow[dd,"{F(\id{},g)}"] \arrow[ld,Rightarrow,"{\alpha_{(B,Y),g}}"{pos=0.61},shorten <=4.45ex,shorten >=2.25ex]
				\\
				\HomC{\A}{B'}{B'}\arrow[d,"{\alpha_{(B',Y'),B'}}"'] \arrow[r,"{-\c g}"']\& \HomC{\A}{B}{B'}\arrow[rd,"{\alpha_{(B,Y),B'}}"{inner sep =0.2ex, pos=0.44}]\arrow[ld,Rightarrow,"{\alpha_{(g,\gamma),B'}}"{pos=0.56},shorten <=3.9ex,shorten >=3.35ex]
				\\
				F((B',Y'),B') \arrow[rr,"{F((g,\gamma),\id{})}"']\&\&F((B,Y),B')
			\end{cd}
		\end{eqD}
		
		Now, given $\eta_B\:1\to F((B,Y),B)$ extraordinary lax natural in $(B,Y)\in \Grothop{M}$, we expand it to functors
		$$\alpha_{(B,Y),A}\:\HomC{\A}{K(B,Y)}{A}\to F((B,Y),A)$$
		$\oplaxn\mbox{-}\lax$ natural in $((B,Y),A)\in {\left(\Grothop{M}\right)}\op\times \A$ as follows, using the cartesian-marked oplax naturality in $(B,Y)$. Given $u\:B\to A$ in $\A$, considering $\underline{u}^{Y}=(u,\id{})$, the structure $2$-cell $\alpha_{(u,\id{}),A}=\id{}$ will give us a commutative square
		\sq[n][3.8][9]{\HomC{\A}{A}{A}}{F((A,M(u)(Y)),A)}{\HomC{\A}{B}{A}}{F((B,Y),A)}{\alpha_{(A,M(u)(Y)),A}}{-\c u}{F((u,\id{}),A)}{\alpha_{(B,Y),A}}
		So, in order to reach the bijection we want, we define
		$$\alpha_{(B,Y),A}(u)\deq F((u,\id{}),A)\left(\eta_{(A,M(u)(Y))}\right).$$
		Given $\theta\:u\aR{} v\:B\to A$ in $\A$, considering
		$$\underline{\theta}^{Y}\:(u,M(\theta)_Y)\aR{}(v,\id{})\:(B,Y)\to (A,M(v)(Y))$$
		and using that $\alpha_{(v,\id{}),A}=\id{}$, we will have by the $2$-dimensional axiom of cartesian-marked oplax naturality that
		$$\alpha_{(B,Y),A}(\theta)=F(\underline{\theta}^Y,A)_{\h\alpha_{(A,M(v)(Y)),A}(\id{A})}\c {\left(\alpha_{(u,M(\theta)_Y),A}\right)}_{\id{A}}$$
		So we firstly define the components of the structure $2$-cells that express the cartesian-marked oplax naturality of $\alpha_{(B,Y),A}$ in $(B,Y)$ and then we will read how to define the action of $\alpha_{(B,Y),A}$ on morphisms $\theta$.
		
		Looking at the diagram of equation~\refs{diagramfromalphatoeta} applied to $(\id{B},\gamma)\:(B,Y)\to (B,Y')$ in $\Grothop{M}$, we see that, in order to have a bijection between the $\alpha$'s and the $\eta$'s, we need to define
		$${\left(\alpha_{(\id{B},\gamma),B}\right)}_{\id{B}}\deq \eta_{(\id{B},\gamma)}.$$
		Whence, given arbitrary $(g,\gamma)\:(B,Y)\to (B',Y')$ in $\Grothop{M}$ and $w\:B'\to A$ in $\A$, since 
		$$(w,\id{})\c (g,\gamma)=(\id{A},M(w)(\gamma))\c (w\c g,\id{}),$$ we need to define
		$${\left(\alpha_{(g,\gamma),A}\right)}_w\deq F((w\c g,\id{}),A)\left(\eta_{(\id{A},M(w)(\gamma))}\right).$$
		And at this point we define, by the argument above,
		$$\alpha_{(B,Y),A}(\theta)\deq F(\underline{\theta}^Y,A)_{\eta_{(A,M(v)(Y))}}\c F((u,\id{}),A)\left(\eta_{(\id{A},M(\theta)_Y)}\right)$$
		for every $\theta\:u\aR{}v\:B\to A$ in $\A$.
		
		Looking at the diagram of equation~\refs{diagramfromalphatoeta} applied to $(g,\id{})\:(B,Y)\to (B',M(g)(Y))$ in $\Grothop{M}$, we see that, in order to have the bijection we want, we need to define
		$${\left(\alpha_{(B,Y),g}\right)}_{\id{B}}\deq \eta_{(g,\id{})}.$$
		Whence, given an arbitrary $f\:A\to A'$ in $\A$ and $u\:K(B,Y)\to A$ in $\A$, by the compatibility axiom of $\oplaxn\mbox{-}\lax$ applied to $(u,\id{})\:(B,Y)\to (A,M(u)(Y))$ in $\Grothop{M}$ and $f\:A\to A'$ in $\A$, we need to define
		$${\left(\alpha_{(B,Y),f}\right)}_u\deq F((u,\id{}),A')\left(\eta_{(f,\id{})}\right).$$
		
		Now, we verify that such assignments work. To show that $\alpha_{(B,Y),A}$ is a functor, consider $u\aR{\theta} v\aR{\rho}w\:B\to A$ in $\A$. Then
		$$\alpha_{(B,Y),A}(\rho\c \theta)\deq F(\underline{(\rho\c \theta)}^Y,A)_{\eta_{(A,M(w)(Y))}}\c F((u,\id{}),A)\left(\eta_{(\id{A},M(\rho\c\theta)_Y)}\right)$$
		while $\alpha_{(B,Y),A}(\rho)\c\alpha_{(B,Y),A}(\theta)$ is equal to
		$$\scalebox{0.935}{$F(\underline{\rho}^Y,A)_{\eta_{(A,M(w)(Y))}}\c F((v,\id{}),A)\left(\eta_{(\id{A},M(\rho)_Y)}\right)\c F(\underline{\theta}^Y,A)_{\eta_{(A,M(v)(Y))}}\c F((u,\id{}),A)\left(\eta_{(\id{A},M(\theta)_Y)}\right)$}$$
		By the extraordinary naturality of $\eta$,
		$$\eta_{(\id{A},M(\rho\c \theta)_Y)}=F((\id{A},M(\theta)_Y),A)\left(\eta_{(\id{A},M(\rho)_Y)}\right)\c \eta_{(\id{A},M(\theta)_Y)}$$
		And by the uniqueness of the liftings of $2$-cells through $\groth{M}$, we have that
		$$\underline{(\rho\c \theta)}^Y=\underline{\rho}^Y\c (\id{A},M(\rho)_Y)\h \underline{\theta}^Y.$$
		So, by $2$-functoriality of $F$, it suffices to prove that
		$$F(\underline{\theta}^Y,A)_{F((\id{A},M(\rho)_Y),A)\left(\eta_{(A,M(w)_Y)}\right)}\c F((u,M(\theta)_Y),A)\left(\eta_{(\id{A},M(\rho)_Y)}\right)$$
		is equal to
		$$F((v,\id{}),A)\left(\eta_{(\id{A},M(\rho)_Y)}\right)\c F(\underline{\theta}^Y,A)_{\eta_{(A,M(v)(Y))}}.$$
		And this is true by naturality of $F(\underline{\theta}^Y,A)$.
		The fact that ${\left(\alpha_{(B,Y),f}\right)}_u$ is a natural transformation is checked with techniques similar to the above ones, noticing that $\underline{(f\h\theta)}^Y=(f,\id{})\h\underline{\theta}^Y$. Whereas showing that ${\left(\alpha_{(g,\gamma),A}\right)}_w$ is a natural transformation uses that for $(g,\gamma)\:(B,Y)\to (B',Y')$
		$$\underline{(\theta\h g)}^{Y}=\underline{\theta}^{M(g)(Y)}\h (g,\id{})\quad \text{and}\quad \underline{\theta}^{Y'}\h(\id{},\gamma)=(\id{},M(v)(\gamma))\h[1.5]\underline{\theta}^{M(g)(Y)}.$$
		At this point, it is straightforward to check that $\alpha_{(B,Y),A}$ is $\oplaxn\mbox{-}\lax$ in $((B,Y),A)$. And we immediately see that we obtain a bijection between the $\alpha$'s and the $\eta$'s.
		
		We extend such bijection to an isomorphism of categories. We send a modification $$\Theta_{(B,Y),A}\:\alpha_{(B,Y),A}\aR{}\beta_{(B,Y),A}\:\HomC{\A}{K(B,Y)}{A}\to F((B,Y),A)$$
		between $\oplaxn\mbox{-}\lax$ natural transformations in $((B,Y),A)$ to the modification
		\begin{cd}[2.5][6.5]
			1 \arrow[r,"{\id{B}}"]\&[-1ex]\HomC{\A}{B}{B}\arrow[r,bend left=24,"{\alpha_{(B,Y),B}}",""'{name=A}]\arrow[r,bend right=24,"{\beta_{(B,Y),B}}"',""{name=B}]\&F((B,Y),B)
			\arrow[Rightarrow,from=A,to=B,shift right=0.5ex,"{\Theta_{(B,Y),B}}"{description,pos=0.46},shorten <=0.15ex,shorten >=0ex]
		\end{cd}
		between extraordinary lax natural transformations. We then send a modification
		$$\Gamma_{(B,Y)}\:\eta_{(B,Y)}\aR{}{\eta'}_{(B,Y)}\:1\to F((B,Y),B)$$
		between extraordinary lax natural transformations to the modification $\Theta$ between $\oplaxn\mbox{-}\lax$ natural transformations such that for every $u\:B\to A$ in $\A$
		$${\left(\Theta_{(B,Y),A}\right)}_{u}\deq F((u,\id{}),A)\left(\Gamma_{(A,M(u)(Y))}\right).$$
		The two assignments are clearly functorial and inverses of each other.
	\end{proof}
	
		We are now ready to show that a pointwise left Kan extension in $\twoCATlax$ along a discrete $2$-opfibration is always a weak left Kan extension.
	
	\begin{prop}\label{pointkanisalsoweak}
		Consider a diagram
		\begin{cd}[5][4]
			\B \arrow[rr,"{F}"]\arrow[d,"{K}"']\&\arrow[dl,Rightarrow,shift right=0.75ex,"{\lambda}"{pos=0.355, inner sep= 0.35ex},shorten <=1.45ex,shorten >= 3ex]\& \C \\
			\A\arrow[rru,"L"']
		\end{cd}
		in $\twoCATlax$ with $\B$ small and $K$ a discrete $2$-opfibration. Assume that $\lambda$ exhibits $L=\Lan{K}{F}$ \pteor{in the sense of \defx\ref{defpointkanextensionlaxthreecat}}. Then $\lambda$ also exhibits $L=\lan{K}{F}$.
	\end{prop}
	\begin{proof}
		Since $L=\Lan{K}{F}$, for every $C$, the $2$-cell
		\begin{equation*}
			\HomC{\A}{K(-)}{A}\aRR{L}\HomC{\C}{(L\c K)(-)}{L(A)}\aRR[\oplaxn]{\HomC{\C}{\lambda_{-}}{\id{}}}\HomC{\C}{F(-)}{L(A)};
		\end{equation*}
		is $2$-universal, giving an isomorphism of categories
		\begin{equation}\label{eqisomorphismpointkanextnelladimcheeweak}
			\HomC{\C}{L(A)}{C}\arr{}\HomC{\moplaxn{\B\op}{\CAT}}{\HomC{\A}{K(-)}{A}}{\HomC{\C}{F(-)}{C}}
		\end{equation}
		
		We need to prove that, for every $U\in \mlax{\A}{\C}$, pasting with $\lambda$ gives an isomorphism of categories
		\begin{equation*}
			\HomC{\mlax{\A}{\C}}{L}{U}\iso \HomC{\mlax{\B}{\C}}{F}{U\c K}.
		\end{equation*}
		The category $\HomC{\mlax{\A}{\C}}{L}{U}$ is isomorphic to the category of $\oplaxn\mbox{-}\lax$ natural transformations
		$$\alpha_{B,A}\:{\HomC{\A}{K(B)}{A}}\aoplaxn{}{\HomC{\C}{F(B)}{U(A)}}$$
		in $(B,A)\in \B\op\times\A$ and modifications between them, thanks to the functoriality of the isomorphism of equation~\refs{eqisomorphismpointkanextnelladimcheeweak}. Indeed, consider a lax natural transformation $\phi\:L\alax{} U$. For every $A\in \A$, the component $\phi_A\:L(A)\to U(A)$ corresponds to a cartesian-marked oplax natural transformation
		$$\alpha_{-,A}\:{\HomC{\A}{K(-)}{A}}\aoplaxn{}{\HomC{\C}{F(-)}{U(A)}}$$
		via the isomorphism of equation~\refs{eqisomorphismpointkanextnelladimcheeweak}. And $\phi_A$ being lax natural in $A\in \A$ precisely corresponds to the cartesian-marked oplax natural transformations $\alpha_{-,A}$ being lax natural in $\A$, with structure $2$-cell on $f\:A\to A'$ in $\A$ given by the image of $\phi_f$ through the isomorphism of equation~\refs{eqisomorphismpointkanextnelladimcheeweak}. And analogously for the modifications.
		
		By \thex\ref{parametrizedyonedaoplaxnlax}, the category $\HomC{\mlax{\A}{\C}}{L}{U}$ is then isomorphic to the category of extraordinary lax natural transformations
		$$1\to \HomC{\C}{F(B)}{U(K(B))}$$
		in $B\in \B$ and modifications between them, which is isomorphic (for example, by Hirata's paper~\cite{hirata_notesonlaxends}) to $\HomC{\mlax{\B}{\C}}{F}{U\c K}$. And we can read that the composite isomorphism is given by pasting with $\lambda$.
	\end{proof}

	\section{Application to the 2-category of elements}\label{sectiontwosetgrothconstr}
	
	In this section, we explore in detail the $2$-category of elements (\defx\ref{defexpltwosetgroth}) from an abstract point of view. Our motivation is to introduce, in future work, an enriched version of Grothendieck fibrations and of the Grothendieck construction. As giving an explicit definition of enriched fibration is quite hard, we would like to capture Grothendieck fibrations and the Grothendieck construction from an abstract point of view and try to generalize such abstract theory. It is also useful to understand the connections between the various properties of the $2$-category of elements (and of the usual Grothendieck construction).
	
	In \thex\ref{grothconstrislaxcomma}, we show that the $2$-category of elements can be captured by a \predfn{lax comma object in $\twoCATlax$}, as defined here in \defx\ref{completeunivproplaxcomma}. This is original, generalizing a known result due to Bird \cite{bird_limitsintwocatoflocpresented} (which appears also in Gray's~\cite{gray_formalcattheory} and Descotte, Dubuc and Szyld's~\cite{descottedubucszyld_sigmalimandflatpseudofun}). The actual difference between our result and Bird's one is that ours allow to consider also lax natural transformations in the $2$-dimensional part of the universal property. This also solves a mismatch in Gray's~\cite{gray_formalcattheory} lax commas. See below for a thorough comparison with the existing literature. 
	
	We explain how our work has potential applications to higher dimensional elementary topos theory. Indeed, in our opinion, the $2$-category of elements should be seen as the archetypal $3$-dimensional classification process, exhibiting $\twoCATlax$ as the archetypal elementary $3$-topos.
	
	In \thex\ref{teorpointkanextforgroth}, we prove a pointwise Kan extension result for the $2$-category of elements, using our original definition of pointwise Kan extension in $\twoCATlax$ (\defx\ref{defpointkanextensionlaxthreecat}). We then show that this result implies many other properties of the $2$-category of elements, also thanks to \prox\ref{pointkanisalsoweak} (a pointwise Kan extension in $\twoCATlax$ is a weak one as well). Among the properties implied, we find the conicalization of weighted $2$-limits (\thex\ref{redlaxnormalconical}) and the $2$-fully faithfulness of the $2$-functor that calculates the $2$-category of elements (that completes \thex\ref{essentialimagegivenbytwosetopf} to $2$-equivalences between $2$-copresheaves and discrete $2$-opfibrations).

	\begin{rem}\label{nextstepislaxcomma}
		\prox\ref{laxnatinsidegrothconstr} gives a cartesian-marked lax natural transformation $\lambda$ of the form 
		\begin{eqD*}
			\sq*[l][5][5][\laxn][2.7][2.2][0.52]{\Grothop{F}}{\1}{\B}{\CAT}{}{\groth{F}}{\1}{F}
		\end{eqD*}
		Bird showed in~\cite{bird_limitsintwocatoflocpresented} that this square exhibits a lax comma. The notion of lax comma used by Bird has been introduced by Gray in~\cite{gray_formalcattheory} and has then been unravelled by Kelly in~\cite{kelly_clubsanddoctrines}. However they did not provide a complete universal property suitable to the lax $3$-categorical ambient $\twoCATlax$. Lambert attempted in~\cite{lambert_discretetwofib} to give a better universal property than the one of Gray and Kelly, but without stating any uniqueness condition in the $2$-dimensional part and only giving a partial $3$-dimensional part. We present in \defx\ref{completeunivproplaxcomma} (see also \prox\ref{laxcommaisobject}) a complete universal property of the lax comma object, that refines both the ones of Gray (and Kelly) and Lambert.
		
		In order to distinguish the explicit definition (given in Gray's~\cite{gray_formalcattheory}) from the complete universal property of the lax comma object, we will call the former ``\predfn{lax comma}" and the latter ``\predfn{lax comma object in $\twoCATlax$}". However, we will use the same symbol for both; this is justified by \prox\ref{laxcommaisobject}.
	\end{rem}
	
	\begin{defne}[Gray~\cite{gray_formalcattheory}]
		Let $F\:\A\to \C$ and $G\:\B\to \C$ be $2$-functors. The \dfn{lax comma from $F$ to $G$} is the $2$-category $\laxcomma{F}{G}$ that is given by the following data:
		\begin{description}
			\item[an object of $\laxcomma{F}{G}$] is a triple $(A,B,h)$ with $A\in \A$, $B\in \B$ and $h\:F(A)\to G(B)$ a morphism in $\C$;
			\item[a $1$-cell $(A,B,h)\to (A',B',h')$ in $\laxcomma{F}{G}$] is a triple $(f,g,\phi)$ with $f\:A\to A'$ in $\A$, $g\:B\to B'$ in $\B$ and
			\sq[l][5][5][\phi][2.3][2.3]{F(A)}{G(B)}{F(A')}{G(B')}{h}{F(f)}{G(g)}{h'}
			a $2$-cell in $\C$;
			\item[a $2$-cell $(f,g,\phi)\aR{}(f',g',\phi')\:(A,B,h)\to (A',B',h')$] is a pair $(\alpha,\beta)$ with $\alpha\:f\aR{}f'$ in $\A$ and $\beta\:g\aR{}g'$ in $\B$ such that
			\twonats[1.95][7.8][6.7][1.05][3.45][2]{F(A)}{G(B)}{F(A')}{G(B')}{F(f')}{F(f)}{G(g')}{G(g)}{h}{h'}{F(\alpha)}{G(\beta)}{\phi}{\phi'}
			\item[the composition] of $1$-cells is given by pasting and that of $2$-cells is inherited by the ones in $\A$ and $\B$.
		\end{description}
		
		The \dfn{oplax comma from $F$ to $G$} is the co of the lax comma from $F\co$ to $G\co$.
	\end{defne}
	
	The following proposition shows the partial universal property of the lax comma object presented by Gray in~\cite{gray_formalcattheory}.
	
	\begin{prop}[Gray~\cite{gray_formalcattheory}]\label{univproplaxcommagray}
		Let $F\:\A\to \C$ and $G\:\B\to \C$ be $2$-functors. The lax comma from $F$ to $G$ is equivalently given by the enriched conical limit in $\twoCAT$ of the diagram
		\begin{cd}[3][3.5]
			\A \arrow[rd,"{F}"']\&\& \ARoplax{\C} \arrow[rd,"{\cod}"] \arrow[ld,"{\dom}"']\&\& \B\arrow[ld,"{G}"]\\
			\& \C \&\& \C
		\end{cd}
		where $\ARoplax{\C}$ is the lax comma from $\Id{\C}$ to $\Id{\C}$.
	\end{prop}
	
		But there is a better universal property that the lax comma satisfies. Indeed, it is a \predfn{lax comma object in the lax $3$-category $\twoCATlax$}, as originally defined here.
	
	\begin{defne}\label{completeunivproplaxcomma}
		Let $\Q$ be a lax $3$-category and consider $1$-cells $F\:\A\to \C$ and $G\:\B\to \C$ in $\Q$. The \dfn{lax comma object in $\Q$ from $F$ to $G$} is, if it exists, an object $\laxcomma{F}{G}\in \Q$ together with a $2$-cell
		\begin{cd}[5.8][5.8]
			\laxcomma{F}{G} \arrow[r,"{\partial_0}"]\arrow[d,"{\partial_1}"']\& \A\arrow[d,"{F}"]\arrow[ld,Rightarrow,"{\lambda}",shorten <= 2.7ex, shorten >=2.2ex]\\
			\B\arrow[r,"{G}"']\& \C
		\end{cd}
		in $\Q$ that is universal in the following lax $3$-categorical sense:
		\begin{enum}
			\item for every $2$-cell $\gamma\:F\c P\aR{}G\c Q\:\M\to \C$, there exists a unique $1$-cell $V\:\M\to \laxcomma{F}{G}$ such that
			\begin{eqD*}
				\begin{cd}*[5][5]
					\M\arrow[rrd,bend left,"{P}"]\arrow[rdd,bend right,"{Q}"']\\[-5ex]
					\&[-5ex]\& \A\arrow[d,"{F}"]\arrow[ld,Rightarrow,shift right=3.4ex,"{\gamma}"{pos=0.53},shorten <= 2.2ex, shorten >=1.4ex]\\
					\&\B\arrow[r,"{G}"']\& \C
				\end{cd}\quad = \quad
				\begin{cd}*[6][6]
					\M\arrow[rrd,bend left,"{P}"]\arrow[rdd,bend right,"{Q}"']\arrow[rd,"{V}",dashed,shorten <=-0.2ex,shorten >=-0.2ex]\\[-5ex]
					\&[-5ex]\laxcomma{F}{G} \arrow[r,"{\partial_0}"]\arrow[d,"{\partial_1}"']\& \A\arrow[d,"{F}"]\arrow[ld,Rightarrow,"{\lambda}",shorten <= 2.7ex, shorten >=2.2ex]\\
					\&\B\arrow[r,"{G}"']\& \C
				\end{cd}
			\end{eqD*}
			\item for every $1$-cells $V,W\:\M\to \laxcomma{F}{G}$ and every $3$-cell
			\begin{eqD*}
				\begin{cd}*[6][6]
					\M\arrow[dd,"{W}"']\arrow[rd,"{V}"]\\[-5ex]
					\&[-4ex]\laxcomma{F}{G}\arrow[ld,Rightarrow,"{\Delta}",shorten <=0ex, shorten >=0ex] \arrow[r,"{\partial_0}"]\arrow[dd,"{\partial_1}"']\& \A\arrow[dd,"{F}"]\arrow[ldd,Rightarrow,"{\lambda}",shorten <= 2.7ex, shorten >=2.2ex]\\[-5ex]
					\laxcomma{F}{G}\arrow[rd,"{\partial_1}"']\\[-5ex]
					\&\B\arrow[r,"{G}"']\& \C
				\end{cd} \quad\aMM{\Xi}\quad
				\begin{cd}*[6][6]
					\M\arrow[rr,"{V}"]\arrow[rd,"{W}"']\&[-5ex]\&[-7ex] \laxcomma{F}{G}\arrow[rd,"{\partial_0}"]\arrow[ld,Rightarrow,"{\Gamma}"]\\[-4ex]
					\&\laxcomma{F}{G} \arrow[rr,"{\partial_0}"]\arrow[d,"{\partial_1}"']\&\&[-5ex]\A\arrow[d,"{F}"]\arrow[lld,Rightarrow,"{\lambda}",shorten <= 2.7ex, shorten >=2.2ex]\\
					\&\B\arrow[rr,"{G}"']\&\& \C
				\end{cd}
			\end{eqD*}
			for $2$-cells $\Gamma$ and $\Delta$, there exists a unique $2$-cell $\nu\:V\aR{}W$ such that
			$$\partial_0\h \nu=\Gamma, \quad \partial_1\h \nu=\Delta, \quad \lambda_{\nu}=\Xi;$$
			notice that we are precisely asking that $\Xi$ corresponds to the $3$-cell given by the lax interchange rule in $\Q$ of 
			\begin{cd}[6][6]
				\M\arrow[r,bend left,"{V}"{pos=0.57},""'{name=A}]\arrow[r,bend right,"{W}"'{pos=0.57},""{name=B}]\&\laxcomma{F}{G} \arrow[r,bend left,"{F\c \partial_0}"{pos=0.43},""'{name=C}]\arrow[r,bend right,"{G\c \partial_1}"'{pos=0.43},""{name=D}]\&\C;
				\arrow[from=A,to=B,Rightarrow,"{\nu}"]
				\arrow[from=C,to=D,Rightarrow,"{\lambda}"]
			\end{cd}
			\item for every $2$-cells $\nu,\omega\:V\aR{}W\:\M\to \laxcomma{F}{G}$ and every pair of $3$-cells $\Phi\:\partial_0\h \nu\aM{}\partial_0\h \omega$ and $\Psi\:\partial_1\h \nu\aM{}\partial_1\h \omega$ such that
			\begin{eqD}{tridimpartcompleteunivproplaxcomma}
				\lambda_{\omega}\c
				\begin{cd}*[6.8][6.8]
					\M\arrow[dd,"{W}"']\arrow[rd,"{V}"]\\[-5ex]
					\&[-4ex]\laxcomma{F}{G}\arrow[ld,Rightarrow,bend right,"{\partial_1\h \nu}"'{pos=0.42},""'{name=A},shorten <=0.4ex, shorten >=0.4ex]\arrow[ld,Rightarrow,bend left,"{\partial_1\h \omega}"{pos=0.58},""{name=B},shorten <=0.4ex, shorten >=0.4ex] \arrow[r,"{\partial_0}"]\arrow[dd,"{\partial_1}"]\& \A\arrow[dd,"{F}"]\arrow[ldd,Rightarrow,"{\lambda}",shorten <= 2.7ex, shorten >=2.2ex]\\[-5ex]
					\laxcomma{F}{G}\arrow[rd,"{\partial_1}"']\\[-5ex]
					\&\B\arrow[r,"{G}"']\& \C
					\arrow[from=A,to=B,triple,"{\Psi}",shorten <=1.1ex,shorten >=0.9ex]
				\end{cd} \quad=\quad
				\begin{cd}*[6.8][6.8]
					\M\arrow[rr,"{V}"]\arrow[rd,"{W}"']\&[-5ex]\&[-7ex] \laxcomma{F}{G}\arrow[rd,"{\partial_0}"]\arrow[ld,Rightarrow,bend right,"{\partial_0\h \nu}"'{pos=0.58},""'{name=C},shorten <=0.4ex,shorten >=0.4ex]\arrow[ld,Rightarrow,bend left,"{\partial_0\h\omega}"{pos=0.42},""{name=D},shorten <=0.4ex,shorten >=0.4ex]\\[-4ex]
					\&\laxcomma{F}{G} \arrow[rr,"{\partial_0}"']\arrow[d,"{\partial_1}"']\&\&[-5ex]\A\arrow[d,"{F}"]\arrow[lld,Rightarrow,"{\lambda}",shorten <= 2.7ex, shorten >=2.2ex]\\
					\&\B\arrow[rr,"{G}"']\&\& \C
					\arrow[from=C,to=D,triple,"{\Phi}",shorten <=1.1ex,shorten >=0.9ex]
				\end{cd}\c \lambda_{\nu}
			\end{eqD}
			there exists a unique $3$-cell $\Theta\:\nu\aM{}\omega$ such that $\partial_0\h \Theta=\Phi$ and $\partial_1\h \Theta=\Psi$.
		\end{enum}
	\end{defne}

	\begin{prop}\label{laxcommaisobject}
		Let $F\:\A\to \C$ and $G\:\B\to \C$ be $2$-functors. There is a lax natural transformation
		\begin{cd}[5.8][5.8]
			\laxcomma{F}{G} \arrow[r,"{\partial_0}"]\arrow[d,"{\partial_1}"']\& \A\arrow[d,"{F}"]\arrow[ld,Rightarrow,"{\lambda}",shorten <= 2.7ex, shorten >=2.2ex]\\
			\B\arrow[r,"{G}"']\& \C
		\end{cd}
		that exhibits the lax comma $\laxcomma{F}{G}$ as the lax comma object in $\twoCATlax$ from $F$ to $G$. Moreover, in the condition $(ii)$ of lax comma object in $\twoCATlax$, if $\Gamma$ and $\Delta$ \pteor{that can be lax natural} are both strict $2$-natural \pteor{resp. pseudo-natural} then also $\nu$ is so.
	\end{prop}
	\begin{proof}
		Firstly, we construct $\lambda$. Given $(A,B,h)\in \laxcomma{F}{G}$, we define the component of $\lambda$ on it to be $h$. Given a morphism $(f,g,\phi)\:(A,B,h)\to (A',B',h')$ in $\laxcomma{F}{G}$, we define the structure $2$-cell of $\lambda$ on it to be $\phi$. It is then straightforward to show that $\lambda$ is a lax natural transformation.
		
		For condition $(i)$ of lax comma object in $\twoCATlax$, since $\lambda$ picks the third component of objects and morphisms in $\laxcomma{F}{G}$ and the other two components are determined by the projections through $\partial_0$ and $\partial_1$, we have to define $V$ as
		$$V(M)\deq (P(M),Q(M),\gamma_M)$$
		$$V(m)\deq (P(m),Q(m),\gamma_m)$$
		for every $M\in \M$ and every morphism $m$ in $\M$. This is easily checked to be a $2$-functor, and it works by construction.
		
		For $(ii)$, take an arbitrary $\Xi$ as above. We see that the three requests
		$$\partial_0\h \nu=\Gamma, \quad \partial_1\h \nu=\Delta, \quad \lambda_{\nu}=\Xi$$
		force us to construct the component of $\nu$ on an arbitrary $M\in \M$ to be the morphism $(\Gamma_M,\Delta_M,\Xi_M)$ in $\laxcomma{F}{G}$. Then the structure $2$-cell of $\nu$ on a morphism $m\:M\to M'$ in $\M$, being a $2$-cell in $\laxcomma{F}{G}$, is determined by its projections through $\partial_0$ and $\partial_1$. So we are forced to define $\nu_m$ to be the $2$-cell $(\Gamma_m,\Delta_m)$ in $\laxcomma{F}{G}$. This is indeed a $2$-cell since $\Xi$ is a modification. It is straightforward to check that $\nu$ is a lax natural transformation, since $\Gamma$ and $\Delta$ are so. And we immediately see that if both $\Gamma$ and $\Delta$ are strict $2$-natural (resp. pseudo-natural) then also $\nu$ is so. The observation that $\nu$ is then the unique lax natural transformation $V\alax{}W$ such that the modification corresponding to the lax interchange rule in $\twoCATlax$ of $\nu$ and $\lambda$ coincides with $\Xi$ follows from \remx\ref{needof2catlax}. 
		
		For $(iii)$, let $M\in \M$. Since the component $\Theta_M$ will be a $2$-cell in $\laxcomma{F}{G}$, it is determined by its projections through $\partial_0$ and $\partial_1$. So we need to define
		$$\Theta_M\deq (\Phi_M,\Psi_M).$$
		That this is indeed a $2$-cell $\nu_M\aR{}\omega_M$ in $\laxcomma{F}{G}$ is guaranteed by equation~\refs{tridimpartcompleteunivproplaxcomma}, taking components on $M$. The $\Theta_M$'s organize into a modification thanks to the fact that $\Phi$ and $\Psi$ are modifications.
	\end{proof}
	
	\begin{rem}
		Notice from \defx\ref{completeunivproplaxcomma} that the lax comma object in a lax $3$-category really is an upgrade of the comma object to a lax $3$-dimensional ambient. Indeed, a lax comma object in a $2$-category is precisely a comma object, since any $\Xi$ of \defx\ref{completeunivproplaxcomma} is then forced to be the identity, and the tridimensional part becomes trivial.

		Interestingly, the uniqueness in the $2$-dimensional part of the universal property of the lax comma object in a lax $3$-category is obtained by considering the lax interchange rule.
		
		The universal property of \prox\ref{univproplaxcommagray} is obtained precisely by restricting ourselves to consider as $\Gamma$ and $\Delta$ only strict $2$-natural transformations.
	\end{rem}
	
	\begin{rem}
		The following theorem is original, refining the result of Bird in~\cite{bird_limitsintwocatoflocpresented} that the $2$-category of elements is given by a lax comma. The actual difference between our result and Bird's one is that ours allows to consider lax natural transformations $\Gamma$ and $\Delta$ in the $2$-dimensional part of the universal property. Moreover lax comma objects in $\twoCATlax$ solve the mismatch of Gray's lax commas between the use of lax natural transformations and the strict ambient $\twoCAT$ that hosts Gray's universal property.
	\end{rem}
	
	\begin{teor}\label{grothconstrislaxcomma}
		Let $F\:\B\to \CAT$ be a $2$-functor. The $2$-category of elements is equivalently given by the lax comma object
		\begin{eqD}{grothconstrislaxcommadiagram}
			\sq*[l][5][5][\lax][2.7][2.2][0.53]{\Grothop{F}}{\1}{\B}{\CAT}{}{\groth{F}}{\1}{F}
		\end{eqD}
		in $\twoCATlax$, exhibited by the cartesian-marked lax natural transformation of \prox\ref{laxnatinsidegrothconstr}. As a consequence, it is then also given by the strict $3$-pullback in $\twoCATlax$ between $F$ and the replacement $\tau$ of $\1\:\1\to \CAT$ obtained by taking the lax comma object of $\1\:\1\to \CAT$ along the identity of $\CAT$ \pteor{that is a lax $3$-dimensional version of the lax limit of the arrow $\1\:\1\to \CAT$}:
		\begin{eqD*}
			\begin{cd}*[5][5]
				\Grothop{F}\PB{rd} \arrow[d,"\groth{F}"'] \arrow[r] \& \CAT\bl \arrow[d,"{\tau}"'] \arrow[r]\& \1 \arrow[d,"{\1}"]\arrow[ld,Rightarrow,shorten <=2.7ex,shorten >=2.2ex,"\lax"{pos=0.53}] \\
				\B\arrow[r,"F"'] \& \CAT \arrow[r,equal]\& \CAT
			\end{cd}
		\end{eqD*}
		The domain of $\tau$ is a lax pointed version of $\CAT$, whence the notation $\CAT\bl$.
	\end{teor}
	\begin{proof}
		The proof is a straightforward calculation. The fact that $\groth{F}$ is then also the strict $3$-pullback of $\tau$ is readily checked by showing that such strict $3$-pullback satisfies the universal property of the lax comma object $\laxcomma{\1}{F}$ in $\twoCATlax$ that we have presented in \defx\ref{completeunivproplaxcomma}, using the universal properties of $\laxcomma{\1}{\Id{\CAT}}$ and of the strict $3$-pullback together with some basics of the calculus of pasting.
	\end{proof}
	
	\begin{rem}\label{remapplelementarythreetopos}
		\thex\ref{grothconstrislaxcomma} shows a potential application of this work to $3$-dimensional elementary topos theory. After \thex\ref{grothconstrislaxcomma}, we should indeed view $\twoCATlax$ as the archetypal example of a would-be notion of elementary $3$-topos. Its classification process is the $2$-category of elements, generalizing Weber's idea in~\cite{weber_yonfromtwotop} that the category of elements is the archetypal $2$-dimensional classification process. Towards a definition of elementary $3$-topos, one can choose between two ways. We can either regulate the classification process with lax comma objects in a lax $3$-category (as originally defined here in \defx\ref{completeunivproplaxcomma}) or take pullbacks along discrete $2$-opfibrations (that serve as replacement).
	\end{rem}
	
	\begin{prop}\label{twosetgrothcanonicallyextendstwofunctor}
		By \thex\ref{grothconstrislaxcomma}, the $2$-category of elements construction canonically extends, for every $2$-category $\B$, to a $2$-functor
		$$\groth{-}\:\mlax{\B}{\CAT}\to \slice{\twoCAT}{\B}$$
	\end{prop}
	\begin{proof}
		Given a lax natural transformation $\phi\:F\aR{}G\:\B\to \CAT$, we define $\groth{\phi}$ as the unique morphism $\groth{\phi}\:\Grothop{F}\to \Grothop{G}$ induced by the universal property of the lax comma object $\Grothop{G}$ in $\twoCATlax$ applied to the lax natural transformation
		\begin{cd}[5.6][6.6]
			{\Grothop{F}}\arrow[d,"{\groth{F}}"'] \arrow[r,"{}"]\& {\1}\arrow[ld,Rightarrow,shift right=1ex,"{\lambda^F}"'{pos=0.475, inner sep=0.25ex},shorten <=3.2ex, shorten >=3.3ex] \arrow[d,"{\1}"]\\
			{\B} \arrow[r,bend left,"{F}"{pos=0.64},""'{name=A}]\arrow[r,bend right,"{G}"',""{name=B}] \& {\CAT}
			\arrow[Rightarrow,from=A,to=B,shorten <=0.4ex,shorten >=0.4ex,"\phi"]
		\end{cd}
		where $\lambda^F$ is the lax natural transformation that presents $\Grothop{F}$ as a lax comma object in $\twoCATlax$. Explicitly, for every $2$-cell $\delta\:(f,\alpha)\aR{}(g,\beta)\:(B,X)\to (C,X')$ in $\Grothop{F}$
		$$\groth{\phi}(B,X)=(B,\phi_B(X)) \quad \text{and} \quad \groth{\phi}(f,\alpha)=(f,\phi_C(\alpha))\quad \text{and} \quad \groth{\phi}(\delta)=\delta.$$ 
		
		Given a modification $\Theta\:\phi\aM{} \psi\:F\aR{}G\:\B\to \CAT$, we define $\groth{\Theta}$ as the unique $2$-natural transformation induced by the universal property of the lax comma object $\Grothop{G}$ in $\twoCATlax$ applied, in the notation of \defx\ref{completeunivproplaxcomma} to $V=\groth{\phi}$, $W=\groth{\psi}$, $\Gamma=\id{}$, $\Delta=\id{}$ and $\Xi$ given by
		\begin{cd}[6.7][7.8]
			{\Grothop{F}}\arrow[d,"{\groth{F}}"'] \arrow[r,"{}"]\& {\1}\arrow[ld,Rightarrow,shift right=1ex,"{\lambda^F}"'{pos=0.475},shorten <=3.2ex, shorten >=3.3ex] \arrow[d,"{\1}"]\\
			{\B} \arrow[r,bend left=35,"{F}"{pos=0.64},""'{name=A}]\arrow[r,bend right=35,"{G}"',""{name=B}] \& {\CAT}
			\arrow[Rightarrow,from=A,to=B,bend right=40,shift right=1ex,shorten <=0.4ex,shorten >=0.4ex,"\phi"',""{name=L}]
			\arrow[Rightarrow,from=A,to=B,bend left=40,shift left=1ex,shorten <=0.4ex,shorten >=0.4ex,"\psi",""'{name=R}]
			\arrow[triple,from=L,to=R,shorten <=0ex,shorten >=-0.2ex,"\Theta"{inner sep=1ex}]
		\end{cd}
		Explicitly, the component of $\groth{\Theta}$ on an object $(B,X)\in \Grothop{F}$ is
		$$\groth{\Theta}_{(B,X)}=\left(\id{B},\Theta_{B,X}\right).$$
		It is straightforward to show that $\groth{-}$ is indeed a $2$-functor.
	\end{proof}
	
	\begin{rem}\label{coopflavourstwosetenrichedgroth}
		The following table shows the four co-op versions of the $2$-category of elements construction, with the corresponding notions of fibration. The first two are given by lax comma objects, while the last two are given by oplax comma objects.
		\begin{center}\h[-18]
			\begin{minipage}{3.3cm}\v[-0.85]
				\begin{eqD*}
					\h[-7]\sq*[l][5.5][4.5][\lax][2.7][2.2][0.53]{\Grothop{F}}{\1}{\B}{\CAT}{}{\groth{F}}{\1}{F}
				\end{eqD*}
				\begin{center}{\footnotesize disc $2$-opfibrations:}
					
					\noindent {\footnotesize opfibrations, locally discrete fibrations}
				\end{center}
			\end{minipage}\quad
			\begin{minipage}{3.38cm}
				\begin{eqD*}
					\h[-7]\sq*[o][5.5][4.5][\lax][2.2][2.7][0.47]{\Groth{F}}{\1}{\B}{\CAT\op}{}{\groth{F}}{\1}{F}
				\end{eqD*}
				\begin{center}{\footnotesize disc $2$-fibrations:}
					
					\noindent {\footnotesize fibrations, locally discrete opfibrations}
				\end{center}
			\end{minipage}\quad
			\begin{minipage}{3.47cm}
				\begin{eqD*}
					\h[-7]\sq*[l][5.5][3.4][\oplax][2.7][2.2][0.53]{\Grothcoop{F}}{\1}{\B}{\CAT\co}{}{\groth{F}}{\1}{F}
				\end{eqD*}
				\begin{center}{\footnotesize disc $2$-coopfibrations:}
					
					\noindent {\footnotesize opfibrations, locally discrete opfibrations}
				\end{center}
			\end{minipage}\quad
			\begin{minipage}{3.53cm}
				\begin{eqD*}
					\h[-7]\sq*[o][5.5][3.3][\oplax][2.2][2.7][0.47]{\Grothco{F}}{\1}{\B}{\CAT\coop}{}{\groth{F}}{\1}{F}
				\end{eqD*}
				\begin{center}{\footnotesize disc $2$-cofibrations:}
					
					\noindent {\footnotesize fibrations, locally discrete fibrations}
				\end{center}
			\end{minipage}
		\end{center}
	\end{rem}	
	
		We now apply Section~\ref{sectionkanextensions} to the $2$-category of elements. We show that the same filled square that exhibits a lax comma object in $\twoCATlax$ also exhibits a pointwise left Kan extension in $\twoCATlax$ (\defx\ref{defpointkanextensionlaxthreecat}). Such result is original.
	
	\begin{teor}\label{teorpointkanextforgroth}
		Let $F\:\A\to \CAT$ be a $2$-functor with $\A$ a small $2$-category. Then the lax comma object square in $\twoCATlax$
		\sq[l][5][5][\lax][2.7][2.2][0.53]{\Grothop{F}}{\1}{\A}{\CAT}{}{\groth{F}}{\1}{F}
		exhibits
		$$F=\Lan{\groth{F}}{\Delta 1}.$$
	\end{teor}
	\begin{proof}
		By \prox\ref{laxnatinsidegrothconstr}, we know that the lax natural transformation $\lambda$ that presents the lax comma object in $\twoCATlax$ is cartesian-marked lax. Given $A\in \A$ and $C\in \CAT$, we prove that the cartesian-marked oplax natural transformation
		\begin{equation*}
			\HomC{\A}{\groth{F}(-)}{A}\aRR{F}\HomC{\CAT}{(F\c \groth{F})(-)}{F(A)}\aRR[\oplaxn]{\HomC{\CAT}{\lambda_{-}}{\id{}}}\HomC{\CAT}{\Delta 1(-)}{F(A)},
		\end{equation*}
		that we call $\mu$, is $2$-universal. Explicitly, $\mu$ has components
		\begin{fun}
			\mu_{(B,X)} & \: & \HomC{\B}{B}{A} & \too & \HomC{\CAT}{1}{F(A)} \\[0.9ex]
			&& \tc*[6]{B}{A}{u}{v}{\theta} &\mtoo & \tc*[8][30][pos=0.59][pos=0.59][,shift right=0.4ex]{1}{F(A)}{F(u)(X)}{F(v)(X)}{F(\theta)_X}
		\end{fun}
		for every $(B,X)\in \Grothop{F}$ and structure $2$-cells
		$${\left(\mu_{(g,\gamma)}\right)}_u=F(u)(\gamma)\:F(u\c g)(X')\to F(u)(X)$$
		on every $(g,\gamma)\:(B,X)\al{} (B',X')$ in $\Grothop{F}$, for every $u\:B\to A$ in $\A$. Given
		$$\sigma\:\HomC{\A}{\groth{F}(-)}{A}\aoplaxn{}\HomC{\CAT}{\Delta 1(-)}{C},$$
		we prove that there exists a unique functor $s\:F(A)\to C$ such that
		$$(s\c -)\c \mu =\sigma.$$
		We see that there is at most one such $s$, as we need, for every $\alpha\:X\to X'$ in $F(A)$,
		$$s(X)=s\left(\mu_{(A,X)}(\id{A})\right)=\sigma_{(A,X)}(\id{A})$$
		$$s(\alpha)=s\left({\left(\mu_{(\id{A},\alpha)}\right)}_{\id{A}}\right)=\left(\sigma_{(\id{A},\alpha)}\right)_{\id{A}}.$$
		And this $s$ works thanks to the fact that $\sigma$ is cartesian-marked oplax.
		
		We now prove the $2$-dimensional universality of $\mu$. Given $$\Xi\:\sigma\aM{}\sigma'\:\HomC{\A}{\groth{F}(-)}{A}\aoplaxn{}\HomC{\CAT}{\Delta 1(-)}{C},$$
		we prove that there exists a unique natural transformation $\xi\:s\aR{} s'\:F(A)\to C$ such that $(\xi\ast -)\h \mu =\Xi$. We see that there is at most one such $\xi$, as we need
		$$\xi_X=\xi_{\mu_{(A,X)}(\id{A})}=\Xi_{(A,X),\id{A}}$$
		for every $X\in F(A)$. And this $\xi$ works. We have thus shown that $\mu$ is $2$-universal.
	\end{proof}
	
	\begin{rem}
		Thanks to \prox\ref{pointkanisalsoweak}, we obtain as a corollary that the $2$-category of elements also exhibits a weak left Kan extension in $\twoCATlax$. The isomorphism of categories that presents such weak left Kan extension has been proved by Bird in~\cite{bird_limitsintwocatoflocpresented}. It also appears in Gray's~\cite{gray_formalcattheory} and in Descotte, Dubuc and Szyld's~\cite[Remark 1.2.4]{descottedubucszyld_sigmalimandflatpseudofun}. Moreover, such isomorphism restricts to different flavours of laxness. All these isomorphisms are a particular case of Proposition~3.18 of Szyld's~\cite{szyld_limlifttheortwodimmonad}.
	\end{rem}

	\begin{coroll}[Bird~\cite{bird_limitsintwocatoflocpresented}, Szyld~\cite{szyld_limlifttheortwodimmonad}]\label{isomorphismsweakkanextension}
		Let $F\:\A\to \CAT$ be a $2$-functor. Then $$F=\lan{\groth{F}}{\Delta 1}.$$
		Moreover the isomorphism of categories 
		$$\HomC{\mlax{\A}{\CAT}}{F}{U}\iso \HomC{\mlax{\Grothopdiag{F}}{\CAT}}{\Delta 1}{U\c \groth{F}},$$
		natural in $U\:\A\to \CAT$, that presents the weak left Kan extension in $\twoCATlax$ \pteor{see \remx\ref{remhomtwocategoriestwocatlax}} restricts to isomorphisms
		$$\HomC{\mps{\A}{\CAT}}{F}{U}\iso \HomC{\msigma{\Grothopdiag{F}}{\CAT}}{\Delta 1}{U\c \groth{F}}$$
		$$\HomC{\m{\A}{\CAT}}{F}{U}\iso \HomC{\mlaxn{\Grothopdiag{F}}{\CAT}}{\Delta 1}{U\c \groth{F}},$$
		where $\opn{ps}$ means to restrict to pseudonatural transformations and $\opn{sigma}$ means to restrict to Descotte, Dubuc and Szyld's~\cite{descottedubucszyld_sigmalimandflatpseudofun} {sigma natural transformations} \pteor{which are a pseudo version of cartesian-marked lax natural transformations}.
	\end{coroll}
	\begin{proof}
		Clear by \prox\ref{pointkanisalsoweak}. It is straightforward to see that the restrictions hold.
	\end{proof}
	
	\begin{rem}
		The third isomorphism of \corx\ref{isomorphismsweakkanextension} offers a shorter but less elementary proof to \thex\ref{redlaxnormalconical} (reduction of weighted $2$-limits to cartesian-marked lax conical ones). Indeed this is what Street showed in~\cite{street_limitsindexedbycatvalued}.
	\end{rem}
	
	\begin{rem}
		We can also deduce the $2$-fully faithfulness of the $2$-category of elements construction (in three laxness flavours) and extend Lambert's \thex\ref{essentialimagegivenbytwosetopf} to $2$-equivalences between $2$-copresheaves and discrete $2$-opfibrations. The fact that the first $2$-functor $\groth{-}$ of \thex\ref{twofullyfaithfulnessgroth} is $2$-fully faithful is proved also in Bird's~\cite{bird_limitsintwocatoflocpresented}, but we show that it is a consequence of the weak Kan extension result. None of the three $2$-equivalence results of \thex\ref{twofullyfaithfulnessgroth} seems to appear in the literature.
	\end{rem}
	
	\begin{teor}\label{twofullyfaithfulnessgroth}
		Let $\A$ be a $2$-category. The $2$-category of elements construction \pteor{extended to consider lax natural transformations as in \prox\ref{twosetgrothcanonicallyextendstwofunctor}} produces a $2$-equivalence
		$$\groth{-}\:\mlax{\A}{\CAT}\aequi \Fib[t][n][\A]$$
		where $\Fib[t][n][\A]$ is the full sub-$2$-category of $\h[1]\slice{\twoCAT}{\A}$ given by the split discrete $2$-opfibrations with small fibres. Moreover this restricts to $2$-equivalences
		$$\groth{-}\:\mps{\A}{\CAT}\aequi \Fib'[t][n][\A]$$
		$$\groth{-}\:\m{\A}{\CAT}\aequi \Fib"[t][n][\A]$$
		where $\Fib'[t][n][\A]$ and $\Fib"[t][n][\A]$ restrict $\Fib[t][n][\A]$ respectively\v[0.4] to cartesian functors \pteor{as underlying functors of the $1$-cells} and to cleavage preserving functors.
	\end{teor}
	\begin{proof}
		We already know by \thex\ref{essentialimagegivenbytwosetopf} that the essential image of $\groth{-}$ (in each of the three versions) is given by the split discrete $2$-opfibrations. So we are missing the $2$-fully faithfulness of the three $2$-functors.
		
		Let $F,G\:\A\to \CAT$ be $2$-functors. Then combining \corx\ref{isomorphismsweakkanextension} (weak Kan extension result) and \thex\ref{grothconstrislaxcomma} (the $2$-category of elements exhibits a lax comma object in $\twoCATlax$) we obtain the composite isomorphism of categories
		$$\HomC{\mlax{\A}{\CAT}}{F}{G}\iso \HomC{\mlax{\Grothopdiag{F}}{\CAT}}{\Delta 1}{G\c \groth{F}}\iso \HomC{\slice{\twoCAT}{\A}}{\Grothopdiag{F}}{\Grothopdiag{G}}$$
		where the first functor is given by pasting with the $2$-cell $\lambda^F$ that presents the lax comma object $\Grothop{F}$ and the second functor is the inverse of pasting with $\lambda^G$ on objects and producing the modification associated to the lax interchange rule (see \remx\ref{needof2catlax}) on morphisms. Indeed the second functor is surely a bijection on objects, and it is also a bijection on morphisms by part $(ii)$ of \defx\ref{completeunivproplaxcomma} with $\Gamma=\id{}$ and $\Delta=\id{}$. Since the composite functor precisely coincides with the functor on morphisms associated to $\groth{-}$ between $\A$ and $\C$ (see the proof of \prox\ref{twosetgrothcanonicallyextendstwofunctor}), this completes the proof of the first $2$-equivalence.
		
		The composite isomorphism above then restricts to the following two:
		$$\scalebox{0.95}{$\HomC{\mps{\A}{\CAT}}{F}{G}\iso \HomC{\msigma{\Grothopdiag{F}}{\CAT}}{\Delta 1}{G\c \groth{F}}\iso \HomC{\Fib'[t][n][\A]}{\Grothopdiag{F}}{\Grothopdiag{G}}$}$$
		$$\scalebox{0.975}{$\HomC{\m{\A}{\CAT}}{F}{G}\iso \HomC{\mlaxn{\Grothopdiag{F}}{\CAT}}{\Delta 1}{G\c \groth{F}}\iso \HomC{\Fib"[t][n][\A]}{\Grothopdiag{F}}{\Grothopdiag{G}}$}$$
		by part $(i)$ of \defx\ref{completeunivproplaxcomma}, since whiskering $\lambda^G$ on the left with a $2$-functor $\Grothop{F}\to \Grothop{G}$ looks at the second component of the morphisms in $\Grothop{G}$.
	\end{proof}

	\subsection*{Acknowledgements}
	
	I would like to thank Francesco Dagnino for the interesting discussions that brought to the idea of $2$-$\V$-enrichment. I thank my supervisor Nicola Gambino for the interesting discussions and his helpful advice on the subject of this paper. Finally, I thank the anonymous referees for their useful comments and suggestions.
	
	\bibliographystyle{abbrv}
	\bibliography{Bibliography2}

\end{document}